\documentclass[11pt]{article}
\usepackage{amsmath,amsthm,amssymb,a4wide} 
\usepackage{amscd, mathrsfs, wasysym, color}  
\usepackage{graphicx}
\usepackage{epsfig} 
\usepackage{slashed}  
\usepackage[hidelinks]{hyperref}
\usepackage[most]{tcolorbox} 
\newtheorem{theorem}{Theorem}[section]
\newtheorem{proposition}[theorem]{Proposition}
\newtheorem{lemma}[theorem]{Lemma}
\newtheorem{corollary}[theorem]{Corollary}

\newtheorem{remark}[theorem]{Remark} 
\numberwithin{equation}{section} 
\numberwithin{figure}{section}  
\renewcommand \div {\text{div}\,}

\newcommand \la \langle
\newcommand \ra \rangle

\newcommand \Ecal {\mathcal{E}} 
\newcommand \Ocal {\mathcal{O}}

\newcommand \trianglerightNEW \triangleright

\newcommand \auth {\textsc} 

\newcommand \CC {\mathbb C}

\newcommand \Lh {\widehat{L}}

\newcommand \bei {\begin{itemize}}
\newcommand \eei {\end{itemize}}
\newcommand \be {\begin{equation}}
\newcommand \bel {\be\label}
\newcommand \ee {\end{equation}}
\newcommand \del \partial
\newcommand \RR {{\mathbb R}}
\newcommand \mE {\Ecal}
\newcommand \Hcal {\mathcal H}

\newcommand \eps \epsilon 
\newcommand \Lcal	{\mathcal L}

\newcommand \tildeA {\widetilde A}
\newcommand \tildechi {\widetilde \chi}

\let\oldmarginpar\marginpar
\renewcommand\marginpar[1]{\-\oldmarginpar[\raggedleft\footnotesize #1]%
{\raggedright\footnotesize #1}}
\begin{document}

\title{Global evolution of the U(1) Higgs Boson: 
nonlinear stability and uniform energy bounds} 

\author{Shijie Dong\footnote{Laboratoire Jacques-Louis Lions, Centre National de la Recherche Scientifique, Sorbonne Universit\'e, 4, Place Jussieu, 75252 Paris, France.  
Email: dongs@ljll.math.upmc.fr, contact@philippelefloch.org}\,, 
\, 
Philippe G. LeFloch$^*$, and Zoe Wyatt\footnote{Maxwell Institute for Mathematical Sciences, School of Mathematics, University of Edinburgh, Edinburgh, EH9 3FD, United Kingddom. Email: zoe.wyatt@ed.ac.uk.}}

\date{February 2019}
\maketitle
 
\begin{abstract} Relying on the hyperboloidal foliation method, we establish the nonlinear stability of the ground state of the $U(1)$ standard model of electroweak interactions. This amounts to establishing a global-in-time theory for the initial value problem for a nonlinear wave-Klein-Gordon system that couples (Dirac, scalar, gauge) massive equations together. In particular, we investigate here the Dirac equation and consider a new energy functional for this field defined with respect to the hyperboloidal foliation of Minkowski spacetime. We provide a novel decay result for the Dirac equation which is uniform in the mass coefficient, and thus allows for the Dirac mass coefficient to be arbitrarily small. Furthermore we obtain energy bounds for the Higgs fields and gauge bosons that are uniform with respect to the hyperboloidal time variable. 
\end{abstract}

\setcounter{secnumdepth}{6}
\setcounter{tocdepth}{1}  

\vfill

\tableofcontents

\vfill

\newpage 


\section{Introduction} 
 
\paragraph*{Main objective.} 

Our primary objective is to study the equations of motion arising from the Higgs mechanism applied to an abelian $U(1)$ gauge theory on a Minkowski background after spontaneous symmetry breaking. We view this model as a stepping-stone towards the full non-abelian Glashow-Weinberg-Salam theory (GSW), also known as the electroweak Standard Model. For background physics information on the models treated in this paper, see for example \cite{Ait-Hey}. 

In short, we study here a class of nonlinear wave equations which involve the first-order Dirac equation coupled to second-order wave or Klein-Gordon equations.  We are interested in the initial value problem for such systems, when the initial data have sufficiently small Sobolev-type norm. 
We provide here a new application of the hyperboloidal foliation method introduced for such coupled systems by LeFloch and Ma \cite{PLF-YM-book}, which has been successfully used to establish 
global-in-time existence results for nonlinear systems of coupled wave and Klein-Gordon equations. 
This method takes its root in pioneering work by Friedrich \cite{Friedrich1,Friedrich2} on the vacuum Einstein equations and by Klainerman \cite{Klainerman80,Klainerman85} and H\"ormander \cite{Hormander} on the (uncoupled) Klein-Gordon equation, as well as Katayama \cite{Katayama12a,Katayama12b}. The issue of nonlinear interaction terms that couple wave and Klein-Gordon equations together, was a challenge tackled in \cite{PLF-YM-book} and in the subsequent developments on the Einstein equations \cite{PLF-YM-cmp, PLF-YM-book2, PLF-YM-arXiv1} and \cite{Fajman,Smulevici}. 

\paragraph*{The model of interest.}

In the {\bf abelian $U(1)$ gauge model}, the set of unknowns consists of a Dirac field $\psi: \RR^{3+1} \to \CC^4$ representing a fermion of mass $m_g$ with spin $1/2$, a vector field $A=(A^\mu)$ representing a massive boson of mass $m_q$ with spin $1$, and a complex scalar field $\chi$ representing the perturbation from the constant minimum $\phi_0$ of the complex Higgs field $\phi = \phi_0+\chi$. Note the set of ground states $\phi_0 = \phi_0(t,x)$ of the theory are defined by $ v^2 := \phi_0^* \phi_0$, however in this work, as in \cite{Tsutsumi2}, we consider only constant ground states satisfying $\partial_\mu \phi_0 = 0$.
 Given such a constant ground state $\phi_0$
and the three physical parameters $m_q, m_\lambda, m_g$, the equations of motion in a modified Lorenz gauge consist of three evolution equations
\begin{subequations} 
\label{eq:intro-u1-pde-modified-lorenz}
\bel{eq:intro-u1-pde-modified-lorenz1}
\aligned
\big( \Box - m_q^2 \big) A^\nu 
& = Q_{A^\nu},
\\
\Box \chi - m_q^2 \, {\phi_0  \over 2 v^2} \big(\phi_0^* \chi - \chi^* \phi_0 \big) 
- m_\lambda^2 \, {\phi_0 \over 2 v^2} \big(\phi_0^* \chi - \chi^* \phi_0 \big)
& = Q_{\chi},
\\
i \gamma^\mu \del_\mu \psi - m_g \psi 
&= Q_{\psi},
\endaligned
\ee
and a constraint equation
\bel{eq:intro-u1-pde-modified-lorenz2}
\aligned
\textrm{div}A + i {m_q \over v \sqrt{2}} \big(\phi_0^* \chi - \chi^* \phi_0 \big)
&=
0. 
\endaligned
\ee
\end{subequations} 
Here, the quadratic nonlinearities $Q_{A^\nu}$, $Q_{\chi}$, and $Q_{\psi}$ are defined later (see \eqref{eq:u1-nl}).
Throughout, we use the signature $(-, +,+, +)$ for the wave operator 
$\Box := \eta^{\mu \nu} \del_\mu \del_\nu  = - \del_t^2 + \Delta$, while the matrices $\gamma^\mu$ and $\gamma_5$ are the standard Dirac matrices (cf.~Section \ref{subsec:dirac-matrices}). 
 
Some important physical parameters are provided together with the mass coefficients above, namely 
\bel{eq:param}
m_q^2 
:= 2 q^2 v^2 > 0, 
\qquad
m_\lambda^2 
:= 4 \lambda v^2 > 0, 
\qquad
m_g 
:= g v^2 \geq 0, 
\ee
themselves depending on given {\bf coupling constants}, denoted by 
$\lambda, g, q$, 
as well as the  ``vacuum expectation value''  of the Higgs field denoted by $v$. 


\paragraph*{The nonlinear stability of the Higgs field.}

Our main result concerning the system \eqref{eq:intro-u1-pde-modified-lorenz} is a proof of the global-in-time existence of solutions for sufficiently small perturbations away from the \textbf{constant vacuum state}, defined by the conditions 
\be 
\begin{split}
A^\mu \equiv 0, 
\quad
\phi \equiv \phi_0 \textrm{ and } \partial_\mu \phi_0 = 0,
\quad
\psi \equiv 0.
\end{split}
\ee
It is convenient to work with the perturbed Higgs field 
$\chi = \phi - \phi_0$ as our main unknown. 
The initial data set are denoted by
\bel{eq:u1-ID}
\big( A^\nu, \chi, \psi \big) (t_0, \cdot) 
= \big( A^\nu_0, \chi_0, \psi_0 \big), 
\qquad
 \big( \del_t A^\nu, \del_t \chi \big) (t_0, \cdot) 
= \big( A^\nu_1, \chi_1 \big), 
\ee
and these data are said to be \textbf{Lorenz compatible} if  
\bel{eq:U1c} 
\aligned
\del_a A_0^a 
&
	= - A^0_1  - i q \big( \phi_0^* \chi_0 - \chi_0^* \phi_0 \big),
\\
\Delta A^0_0 - m_q^2 A^0_0  
&=  - \del_i A_1^i - i q \big( \phi_0^* \chi_1 - \chi_1 ^* \phi_0 \big) + i q \big( \chi_0^* \chi_1 - \chi_1 ^* \chi_0 \big) 
\\
& \quad + 2 q^2 A^0_0 \big( \phi_0^* \chi_0 + \chi_0^* \phi_0 + \chi_0^* \chi_0 \big) + q \psi_0^* \psi_0.
\endaligned
\ee
This elliptic-type system consists of two equations for $11$ functions. In particular, it is easily checked that it admits non-trivial solutions, for instance with compact support.  Observe also that throughout we use the convention that Greek indices take values in $\{0,1,2,3\}$ and Latin indices in $\{1,2,3\}$. In the following statement, we have $m_g \geq 0$, and the coefficient ${m_g^{-1}}$ is interpreted as $+\infty$ when $m_g = 0$.

\begin{theorem}[Nonlinear stability of the ground state for the Higgs boson] 
\label{thm:U1}
Consider the system \eqref{eq:intro-u1-pde-modified-lorenz} with parameters $m_q, m_\lambda > 0$, $m_g \in [0, \min (m_q, m_\lambda)]$ and let $N$ be a sufficiently large integer. 
There exists $\eps_0 > 0$, which is  \underline{independent} of $m_g$, such that for all $\eps \in (0, \eps_0)$ and all compactly supported, Lorenz compatible initial data (in the sense of \eqref{eq:U1c}) satisfying the smallness condition
\bel{eq:ID-small} 
\| A_0, \chi_0, \psi_0 \|_{H^{N+1}(\RR^3)} +  \|A_1, \chi_1 \|_{H^{N}(\RR^3)}  \leq \eps,
\ee
the initial value problem of \eqref{eq:intro-u1-pde-modified-lorenz} admits a global-in-time solution $(A, \chi, \psi )$ with, moreover, 
\bel{eq:mainthmpsiest} 
|A| 
\lesssim 
\eps t^{-3/2}, 
\qquad
|\chi| 
\lesssim 
\eps t^{-3/2},
\qquad
|\psi| 
\lesssim 
\eps \min \big( t^{-1}, m_g^{-1} t^{-3/2} \big).
\ee
Furthermore we obtain the following uniform energy estimates at highest order
\bel{eq:mainthmL2est}
\big\|(s/t) \del^I L^J (\del_\mu A_\nu, \del_\mu \chi) \big\|_{L^2_f(\Hcal_s)}
 \lesssim
	C_1 \eps, 
 \quad |I|+|J|= N,
\ee
and logarithmic energy growth only for the Dirac field
\be
\big\| (s/t) \del^I L^J  \del_\mu \psi \big\|_{L^2_f(\Hcal_s)} 
 \lesssim
	C_1 \eps \log s, 
\quad |I|+|J|= N.
\ee
\end{theorem}

\begin{remark}
The result in Theorem \ref{thm:U1} holds in a fairly straightforward way for the cases $m_g = 0$ or $m_g \simeq \min ( m_q, m_\lambda)$. However much more is required to obtain a result uniform in terms of the mass parameter $m_g \in [0, \min (m_q, m_\lambda)]$.
\end{remark}
\begin{remark} The assumption $m_g \in [0, \min (m_q, m_\lambda)]$ in
Theorem \ref{thm:U1} can be understood both mathematically and physically. 
Consider the Klein-Gordon equation 
$$\Box w - m_g^2 w = F .$$
In order to obtain an estimate of the energy on a constant-time slice $t$, we
use the standard energy estimate $$\int_{\RR^3} \big| \del w(t) \big|^2 +
m_g^2 |w(t)|^2 \, dx \leq \int_{\RR^3} \big| \del w(t_0) \big|^2 + m_g^2
|w(t_0)|^2 \, dx + \int_{t_0}^t \int_{\RR^3} |F(t’) \del_t w(t’)| \, dx
dt’$$
where here $\del \in \{\del_t, \del_x, \del_y, \del_z \}$. Any energy estimate for $t \geq t_0$ clearly requires control over the initial term
$m_g^2 \|w(t_0)\|^2_{L^2(\RR^3)}$. Since this term depends on the mass parameter
$m_g$, we clearly require $m_g$ to have an upper bound in order for any energy argument to be made independent of $m_g$. What the upper bound on $m_g$  should be is clearly motivated from physics. In physical models, $m_q$ and $m_\lambda$ represent the
masses of the gauge and Higgs fields respectively, while $m_g$ represents any fermion masses. Fermion masses are
extremely small compared to both $m_q, m_\lambda$, and so it is particularly relevant to
study the system in the entire range $m_g \in [0, \min (m_q, m_\lambda)]$.
\end{remark}
 
\paragraph*{A simpler model: the Dirac-Proca equations.}

Keeping \eqref{eq:intro-u1-pde-modified-lorenz} in mind, it is interesting to first understand a simplified subset of \eqref{eq:intro-u1-pde-modified-lorenz}, called the Dirac-Proca equations. In the Lorenz gauge the equations of motion for this model read
\bel{eq:intro-dirac-proca} 
\aligned
\Box A^\nu - m^2 A^\nu  
& = 
	-\psi^* \gamma^0 \gamma^\nu (P_L \psi),
\\
-i \gamma^\mu \del_\mu \psi + M \psi 
&= 
	- \gamma^\mu A_\mu (P_L \psi),	
\endaligned 
\ee
where $P_L = {1 \over 2} (I_4 - \gamma^5)$. 
This system describes a spinor field $\psi: \RR^{3+1} \to \CC^4$ representing a fermion of mass $M$ with spin $1/2$ and a vector field $A=(A^\mu)$ representing a massive boson of mass $m$ with spin $1$. We allow the mass parameters $M, m$ to be positive or zero and, in particular, we will be interested in the case $M = 0$ (i.e. the standard Dirac-Proca equations) as well as in the case $M > 0$ (massive field). Observe that the Dirac-Proca system \eqref{eq:intro-dirac-proca} contains no Higgs field and so the masses $M, m$ are introduced ``artificially'' in the model, instead of arising from the Higgs mechanism. 

An initial data set for \eqref{eq:intro-dirac-proca}, say, 
\bel{eq:101-ID} 
\big(A^\nu, \del_t A^\nu, \psi \big) (t_0, \cdot)
=
\big(a^\nu, b^\nu, \psi_0 \big),
\ee
is called Lorenz compatible if one has 
\bel{eq:lorenzc} 
\aligned
b^0 + \del_j a^j 
	& = 0,
\\ 
\Delta a^0 - m^2 a^0 
& = 
	-\psi_0^* P_L \psi_0 - \del_j b^j.
\endaligned
\ee
We postpone the statement of our second main result, for this system, to Theorem \ref{theo-oursecondone} below. 

 
\paragraph*{Strategy of proof.}

The small data global existence problem for \eqref{eq:intro-dirac-proca} with $m_g \equiv 0$ and for \eqref{eq:intro-u1-pde-modified-lorenz} was solved by Tsutsumi in \cite{Tsutsumi, Tsutsumi2}. Nonetheless it is very interesting to revisit this system via the hyperboloidal foliation method developed in \cite{PLF-YM-book} as it provides a different perspective on the problem and sharper estimates, as detailed below. We thus introduce a hyperboloidal foliation which covers the interior of a light cone in Minkowski spacetime, and we then construct the solutions of interest in the future of an initial hyperboloid. For compactly-supported initial data this is equivalent to solving the initial value problem for a standard $t$=constant initial hypersurface. 

Although restricting to compactly supported initial data is a strong assumption, and is not assumed by Tsutsumi in \cite{Tsutsumi, Tsutsumi2}, we find our results nonetheless very interesting. Furthermore the compact support restriction can be removed, subject to some long computational checks, using the Euclidian-Hyperboloidal Foliation Method developed in \cite{PLF-YM-arXiv1,PLF-YM-Comptes}.  

A key new insight in our work that arises from using $L^2$ norms defined on hyperboloids, is that we are able to  derive the following energy functional on hyperboloids
\bel{eq:hyp-dirac-intro}
E^\Hcal(s, \psi) 
:= \int_{\Hcal_s} \left(  \psi^* \psi - \frac{x_i}{t} \psi^* \gamma^0 \gamma^i \psi \right) dx. 
\ee
We can show that this energy is positive definite and controls the norm $\| (s/t) \psi \|_{L^2}$, which provides us with a notion of energy on hyperboloids and should be compared with the more standard approach based on the standard $t=$ constant hypersurfaces; cf.~for example \cite{Bournaveas}. In the case $0 \leq M \ll m$ we use the Dirac equation (expressed in a first-order form and using this energy functional) combined with Sobolev-type estimates. 


\paragraph*{Challenges for the global analysis and new insights.}

There are several difficulties in dealing with the system of coupled wave--Klein--Gordon equations \eqref{eq:intro-u1-pde-modified-lorenz}. 
The first and most well-known one is that the standard Klein-Gordon equation does not commute with the scaling Killing field $S=x^\mu \del_\mu$ of Minkowski spacetime, which prevents us from applying the standard vector-field method in a direct manner. 

Second, the nonlinearities in the Klein--Gordon equation (or the Proca equation) include a bad term  describing $\psi$--$\psi$ interactions, which can be regarded as a wave--wave interaction term and does not have good decay. Tsutsumi \cite{Tsutsumi} was able to overcome this difficulty by defining a new variable whose quadratic nonlinearity only involves the so-called `strong null forms' $Q_{\mu j}$ (first introduced in \cite{Georgiev}) compatible with the scaling vector field. A similar, but complicated new variable, was also defined in \cite{Tsutsumi2}. Due to our use of the hyperboloidal foliation method we do not need to find new variables leading solely to strong null forms, which is an important simplification useful in our future work studying the full GSW theory. 

Next, our global-in-time existence result is established under a lower regularity assumption on the initial data. Namely, the boundedness of the initial data in the norm $\|\cdot \|_{H^N}$ ($N\geq 20$) is needed in \cite{Tsutsumi}, while we only require $N\geq 6$. Most importantly, in the main statement in \cite{Tsutsumi, Tsutsumi2}, a slow growth of the energy of all high-order derivatives occurs. 
By contrast, with our method of proof {\sl we do not have any growth factor} in the $L^2$ norm \eqref{eq:mainthmL2est} and the $L^\infty$ norms of the Higgs scalar field and vector boson field. Note that uniform energy bounds were also established in \cite{PLF-Wei}.   We refer to Section  \ref{sec:bootstrap} for a further discussion of these improved estimates. 

Furthermore, it is challenging to establish a global stability result uniformly for all $m_g \geq 0$, while it is relative easy to complete the proof for either $m_g = 0$ or $m_g$ a large constant. In the case where the mass is small, say $m_g = \eps^2 $, if we treat $\psi$ as a Klein-Gordon field then the field decays like 
$$
|\psi | 
\lesssim 
C_1 \eps^{-1} t^{-3/2}.
$$
However with this decay we cannot arrive at the improved estimate $(1/2) C_1 \eps$ if we start from the \textit{a priori} estimate $C_1 \eps$, where $C_1$ is some large constant introduced in the bootstrap method. This is because the ``improved" estimates we find are $\eps + C_1^2$ instead of $\eps + (C_1 \eps)^2$. Hence $\psi$ behaves more like a wave component when $m_g \ll 1$, but since the mass $m_g$ may be very small but non-zero, we cannot  apply techniques for wave equations.
We find it possible to overcome these difficulties by analysing the first--order Dirac equation, which admits the positive energy functional \eqref{eq:hyp-dirac-intro}, and this energy plays a key role in the whole analysis. Furthermore although we obtain logarithmic growth for the $L^2$ norm of our Dirac field, we obtain interesting $L^\infty$ estimates accessible through our use of the hyperboloidal foliation. We refer to Section \ref{sec:Sobolev-est-Dirac} and Theorem \ref{thm:sup-Dirac1} for a further discussion of these improved estimates. One of the authors has also recently established further results concerning the zero-mass limit of Klein-Gordon equations \cite{Dong19}.

\paragraph*{The coupling constants.} 
For the mass parameters \eqref{eq:param} appearing in the $U(1)$--model
we assume that
\be
m_q \simeq m_\lambda >0,
\qquad
m_g \geq 0, 
\ee 
which depend on the coupling constants $q,g$ and the Higgs constants $\lambda, v$. 
On the other hand, for the mass parameters appearing in the Dirac--Proca model we assume\footnote{There is no Higgs field to generate the masses, and so we do not need to use  subscripts for the mass coefficients.}
$$
m>0, \qquad M \geq 0. 
$$

\paragraph*{Outline of this paper.}  

In Section \ref{sec:hyper-energies}, we will introduce the hyperboloidal foliation method, the Dirac equation, and the energy for the Dirac component.  In Sections \ref{subsec:dirac-matrices} and \ref{subsec:hyperboloid-defs} we give basic definitions for the Dirac equation and hyperboloidal foliation, the latter taken from \cite{PLF-YM-book}. In Section \ref{subsec:dirac-energies} we define the energy for the Dirac field on hyperboloids and give an energy estimate in Proposition \ref{prop-hyperboloidal-energy-inequality-Dirac}. Note this will be complemented by other formulations of the energy functional in Section \ref{sec:choleksy-weyl}. Finally in Section \ref{subsec:2ndorder-dirac-energy} we convert the Dirac equation into a second order wave/Klein-Gordon equation and define the appropriate energy functional. 
Further properties about the hyperboloidal energy functionals are established in Section \ref{sec:choleksy-weyl}. Therein we provide complementing views on the Dirac energy $E^\Hcal$. The first comes from using a Cholesky decomposition for the energy integrand in Section \ref{subsec:choleksy} and next using a Weyl decomposition for the Dirac spinor in Section \ref{subsec:Weyl}.
Then we study the Dirac--Proca model and the abelian model in Section \ref{section:DiracProca} and Section \ref{section:U1model}, respectively.
Next, in Section \ref{sec:nonlinearities} we discuss the system of equations \eqref{eq:intro-u1-pde-modified-lorenz} and its nonlinearities are studied. Finally, in Section~\ref{sec:bootstrap} we rely on a bootstrap argument and prove the desired the stability result.


\section{Hyperboloidal energy functionals for the Dirac operator} \label{sec:hyper-energies}

\subsection{Dirac spinors and matrices} \label{subsec:dirac-matrices}

This section is devoted to analyzing energy functionals for the Dirac equation with respect to a hyperboloid foliation of Minkowski spacetime. Recall the Dirac equation for the unknown $\psi \in \CC^4$
\bel{Dirac1} 
-i \gamma^\mu \del_\mu \psi + M \psi 
= F,
\ee 
with prescribed right-hand side $F \in \CC^4$ and mass $M \in \RR$. To make sense of this equation we need to define various complex vectors and matrices. 
 
For a complex vector $z = (z_0, z_1, z_2, z_3)^T \in \CC^4$ let $\bar{z}$ denote the conjugate, and $z^* := (\bar{z}_0, \bar{z}_1, \bar{z}_2, \bar{z}_3)$ denotes the conjugate transpose. If also $w \in \CC^4$ then the conjugate inner product is defined by
$$
\la z, w \rangle
:=
z^* w 
= \sum_{\alpha=1}^4 \bar{z}_\alpha w_\alpha.
$$
The Hermitian conjugate\footnote{In the physics literature, $A^*$ is often denoted by $A^\dagger$.} 
of a matrix $A$ is denoted by $A^*$, meaning 
$$ 
(A^*)_{\alpha \beta} 
:= (\bar{A})_{\beta \alpha}.
$$
The Dirac matrices $\gamma^\mu$ for $\mu = 0, 1, 2, 3$ are $4 \times 4$ matrices satisfying the identities
\bel{gamma-identities}
\aligned
& \{ \gamma^\mu, \gamma^\nu \} := \gamma^\mu \gamma^\nu + \gamma^\mu \gamma^\nu 
= - 2 \eta^{\mu \nu} I, 
\\
& (\gamma^\mu)^* 
= - \eta_{\mu \nu} \gamma^\nu, 
\endaligned
\ee
where $\eta = \mathrm{diag}(-1, 1, 1, 1)$.  
The Dirac matrices give a matrix representation of the Clifford algebra. We will use the Dirac representation, so that the Dirac matrices take the following form
\bel{Dirac-representation}
\gamma^0 
= \begin{pmatrix} I_2 & 0 \\ 0 & - I_2 \end{pmatrix} 
\,, 
\quad 
\gamma^i 
= \begin{pmatrix} 0 & \sigma^i \\ - \sigma^i & 0 \end{pmatrix}, 
\ee
where $\sigma^i$'s are the standard Pauli matrices: 
\be 
\sigma^1 
:=
\begin{pmatrix}
0 & 1 \\
1 & 0
\end{pmatrix},  \quad
 \sigma^2 
 :=
\begin{pmatrix}
0 & -i \\
i & 0
\end{pmatrix},  \quad
\sigma^3 
:=
\begin{pmatrix}
1 & 0 \\
0 & -1
\end{pmatrix}.
\ee
One often uses the following product of the Gamma matrices
\be 
\gamma_5 
:= i \gamma^0 \gamma^1 \gamma^2 \gamma^3 
= \begin{pmatrix}
0 & I_2 \\
I_2 & 0
\end{pmatrix}.
\ee
The $\gamma_5$ matrix squares to $I_4$ and so we also define the following projection operators 
\bel{eq:proj-ops}
P_L 
:= \frac12 \left( I_4 - \gamma_5 \right),
\qquad
P_R 
:= \frac12 \left( I_4 + \gamma_5 \right)
\ee
to extract the `left-handed' and `right-handed' parts of a spinor, that is, to extract its chiral parts. 
We also note the following useful identities\footnote{In physics, the notation $\bar{\psi} := \psi^\dagger \gamma^0$ is often
 used, but we will avoid this here.}
\be 
(\gamma^\mu)^* 
= \gamma^0 \gamma^\mu \gamma^0,
\quad
(\gamma^0 \gamma^\mu)^* 
= \gamma^0 \gamma^\mu,
\quad 
\{ \gamma_5, \gamma^\mu \} 
= 0. 
\ee 
For further details on Dirac matrices and the representation considered, see \cite{Ait-Hey}.


\subsection{Hyperboloidal foliation of Minkowski spacetime} \label{subsec:hyperboloid-defs}
 
In order to introduce the energy formula of the Dirac component $\psi$ on hyperboloids,
we first need to recall some notation from \cite{PLF-YM-book} concerning the hyperboloidal foliation method. We denote the point $(t, x) = (x^0, x^1, x^2, x^3)$ in Cartesion coordinates, with its spatial radius $r := | x | = \sqrt{(x^1)^2 + (x^2)^2 + (x^3)^2}$. We write $\del_\alpha:= \del_{x^\alpha}$ (for $\alpha=0, 1, 2, 3$) for partial derivatives and 
\be
L_a 
:= x^a \del_t + t \del_a, \qquad a= 1, 2, 3
\ee
for the Lorentz boosts. Throughout the paper, we consider functions defined in the interior of the future light cone $\mathcal{K}:= \{(t, x): r< t-1 \}$, with vertex $(1, 0, 0, 0)$. We consider hyperboloidal hypersurfaces $\mathcal{H}_s:= \{(t, x): t^2 - r^2 = s^2 \}$ with $s>1$. Also $\mathcal{K}_{[s_0, s_1]} := \{(t, x): s_0^2 \leq t^2- r^2 \leq s_1^2; r<t-1 \}$ is used to denote subsets of $\mathcal{K}$ limited by two hyperboloids.

The semi-hyperboloidal frame is defined by
\bel{eq:semi-hyper}
\underline{\del}_0
:= \del_t, \qquad \underline{\del}_a:= {L_a \over t} = {x^a\over t}\del_t+ \del_a.
\ee
Observe that the vectors $\underline{\del}_a$ generate the tangent space to the hyperboloids. We also introduce the vector field $\underline{\del}_{\perp}:= \del_t+ {x^a \over t}\del_a$ which is orthogonal to the hyperboloids.

For this semi-hyperboloidal frame above, the dual frame is given by $\underline{\theta}^0:= dt- {x^a\over t}dx^a$, $\underline{\theta}^a:= dx^a$. The (dual) semi-hyperboloidal frame and the (dual) natural Cartesian frame are connected by the relation
\bel{semi-hyper-Cts}
\underline{\del}_\alpha
= \Phi_\alpha^{\alpha'}\del_{\alpha'}, 
\quad 
\del_\alpha
= \Psi_\alpha^{\alpha'}\underline{\del}_{\alpha'}, 
\quad 
\underline{\theta}^\alpha
= \Psi^\alpha_{\alpha'}dx^{\alpha'}, 
\quad 
dx^\alpha
= \Phi^\alpha_{\alpha'}\underline{\theta}^{\alpha'},
\ee
where the transition matrix ($\Phi^\beta_\alpha$) and its inverse ($\Psi^\beta_\alpha$) are given by
\be
(\Phi_\alpha^{ \beta})
=
\begin{pmatrix}
1 & 0 &   0 &  0   \\
{x^1 / t} & 1  & 0   &  0  \\
{x^2 / t} &  0  &  1  &  0   \\
{x^3 / t} &  0 & 0   & 1
\end{pmatrix}, 
\qquad 
(\Psi_\alpha^{ \beta})
=
\begin{pmatrix}
1 & 0 &   0 &  0   \\
-{x^1 / t} & 1  & 0   &  0  \\
-{x^2 / t} &  0  &  1  &  0   \\
-{x^3 / t} &  0 & 0   & 1
\end{pmatrix}.
\ee

Throughout, we use roman font $E$ to denote energies coming from a first-order PDE (see below) and, calligraphic font $\Ecal$ to denote energies coming from a second-order PDE. In the Minkowski background, we first introduce the energy $\Ecal$ for scalar-valued or vector-valued maps $\phi$ defined on a hyperboloid $\Hcal_s$: 
\bel{eq:2energy} 
\aligned
\Ecal_m(s, \phi)
&:= \int_{\Hcal_s} \Big( |\del_t \phi |^2+ \sum_a |\del_a \phi |^2+ (x^a/t)(\del_t \phi^* \del_a \phi + \del_t \phi \del_a \phi^*)+ m^2 |\phi |^2 \Big) \, dx
\\
               &= \int_{\Hcal_s} \Big( |(s/t)\del_t \phi |^2+ \sum_a |\underline{\del}_a \phi |^2+ m^2 |\phi |^2 \Big) \, dx
                \\
               &= \int_{\Hcal_s} \Big( |\underline{\del}_\perp \phi |^2+ \sum_a |(s/t)\del_a \phi |^2+ \sum_{a<b} |t^{-1}\Omega_{ab} \phi |^2+ m^2 |\phi |^2 \Big) \, dx,
\endaligned
\ee
where 
\be
\Omega_{ab}
:= x^a\del_b- x^b\del_a 
\ee
denotes the rotational vector field. We also write $\Ecal(s, \phi):= \Ecal_0(s, \phi)$ for simplicity.  All of our integrals in $L^1$, $L^2$, etc. are defined from the standard (flat) metric in $\RR^3$, so 
\bel{flat-int}
\|\phi \|_{L^1_f(\Hcal_s)}
=\int_{\Hcal_s}|\phi | \, d S
:=\int_{\RR^3} \big|\phi(\sqrt{s^2+r^2}, x) \big| \, dx.
\ee


\subsection{Hyperboloidal energy of the Dirac equation} 
\label{subsec:dirac-energies}

We now derive a hyperboloidal energy for the Dirac equation \eqref{Dirac1}. Premultiplying the PDE \eqref{Dirac1} by $\psi^* \gamma^0$ gives
\bel{psi-times-Dirac}
\psi^* \del_0 \psi + \psi^* \gamma^0 \gamma^j \del_j \psi + i M \psi^* \gamma^0 \psi
= i \psi^* \gamma^0 F \,.
\ee 
The conjugate of \eqref{Dirac1} is 
$$
(\del_\mu \psi^*) (\gamma^\mu)^* - i \psi^* M 
= -i F^*.
$$
Multiplying this equation by $\psi$ gives
\bel{Dirac*-times-psi}
(\del_0 \psi^*) \psi + (\del_j \psi^*) \gamma^0 \gamma^j \psi - i M \psi^* \gamma^0 \psi 
= -i F^* \gamma^0 \psi.
\ee 
Adding \eqref{psi-times-Dirac} and \eqref{Dirac*-times-psi} together yields
\bel{del-Dirac}
\del_0(\psi^* \psi) + \del_j(\psi^* \gamma^0 \gamma^j \psi) 
=  i\psi^* \gamma^0 F -iF^* \gamma^0 \psi.
\ee
Note the mass term does not appear in \eqref{del-Dirac}. Moreover recalling $2\text{Re}[z] = z+\bar{z}$ for some $z \in \CC$ then we see that \eqref{del-Dirac} is the real part of \eqref{psi-times-Dirac}. It would appear however if we subtracted \eqref{psi-times-Dirac} from \eqref{Dirac*-times-psi}, that is the imaginary part of \eqref{psi-times-Dirac}, then we find
$$ 
\psi^* \del_0 \psi - \del_0 \psi^* \cdot \psi + \psi^* \gamma^0 \gamma^j \del_j \psi - \del_j \psi^* \cdot \gamma^0 \gamma^j \psi + 2iM \psi^* \gamma^0 \psi 
= i \psi^* \gamma^0 F + i F^* \gamma^0 \psi.
$$
However such an expression does not appear to be useful.  We return to \eqref{del-Dirac} and integrate over regions in spacetime to obtain energy inequalities. These give the following two definitions of first-order energy functionals. 

Integrating \eqref{del-Dirac} over $[t_0, t] \times \RR^3$, and assuming spatially compactly supported initial data, gives the following result.

\begin{lemma}
On $t=const$ slices, define the energy 
\bel{energy-flat-Dirac}
E^{\textrm{flat}}(t, \psi) 
:= \int_{\RR^3} (\psi^* \psi) (t, x) dx.
\ee
Then it holds
\be 
E^{\textrm{flat}}(t, \psi) 
= E^{\textrm{flat}}(t_0, \psi) 
+ \int_{t_0}^t \int_{\RR^3} \big(i\psi^* \gamma^0 F -iF^* \gamma^0 \psi\big) \, dxdt. 
\ee 
\end{lemma} 

Such functionals on a constant-time foliation have been considered frequently in the literature, see for instant \cite{D-F-S}. The following Lemma gives a new perspective using a hyperboloidal foliation. 

\begin{lemma}
On hyperboloidal slices $\Hcal_s$ define the energy
\bel{energy-Dirac1}
E^\Hcal(s, \psi) 
:= 
\int_{\Hcal_s} \left(  \psi^* \psi - \frac{x_i}{t} \psi^* \gamma^0 \gamma^i \psi \right) dx. 
\ee
For solutions to \eqref{del-Dirac} this satisfies
\bel{integrated-energy-equality}
E^\Hcal(s, \psi) 
= E^\Hcal(s_0, \psi) 
+ \int_{s_0}^s \int_{\Hcal_{\bar{s}}} (\bar{s}/t) ( i\psi^* \gamma^0 F -iF^* \gamma^0 \psi) \,dx d \bar{s}.
\ee
\end{lemma}

\begin{proof}
Integrating \eqref{del-Dirac} over $\mathcal{K}_{[s_0, s]}$ gives
\be 
\aligned
&\quad
\int_{\Hcal_s} 
\big( \psi^* \psi,  \psi^* \gamma^0 \gamma^j \psi \big) \cdot n \, d \sigma
- \int_{\Hcal_{s_0}} 
\big( \psi^* \psi,  \psi^* \gamma^0 \gamma^j \psi \big) \cdot n \, d \sigma
\\
&=   
\int_{\mathcal{K}_{[s_0, s]}} ( i\psi^* \gamma^0 F -iF^* \gamma^0 \psi) dt dx.
\endaligned
\ee
Here $n$ and $d \sigma$ are the unit normal and induced Lebesgue measure on the hyperboloids respectively
\bel{eq:n-dsigma}
n 
= (t^2+r^2)^{-1/2}(t, -x^i), 
\quad 
d \sigma 
= t^{-1} (t^2+r^2)^{1/2} dx. 
\ee
Using this explicit form of $n$ and $d\sigma$ gives the result. 
\end{proof}

To our knowledge the hyperboloidal Dirac energy \eqref{energy-Dirac1} is new. 
We now show that this energy $E^\Hcal(s, \psi)$ is indeed positive definite.

\begin{proposition}[Hyperboloidal energy of the Dirac equation. I]
\label{eq:positive-energy-1storder}
By defining
\begin{subequations}
\be
E^+(s, \psi) 
:= \int_{\Hcal_s} 
\Big(\psi - \frac{x_i}{t} \gamma^0 \gamma^i \psi \Big)^* 
\Big(\psi - \frac{x_j}{t} \gamma^0 \gamma^j \psi \Big) \, dx, 
\ee
one has
\bel{eq:energy-identity}
E^\Hcal (s, \psi) 
= \frac{1}{2} E^+(s, \psi) + {1 \over 2} \int_{\Hcal_s} {s^2 \over t^2} \psi^* \psi \, dx, 
\ee
which in particular implies the positivity of the energy
\bel{eq:lowbound1} 
E^\Hcal (s, \psi) 
\geq {1 \over 2} \int_{\Hcal_s} {s^2 \over t^2} \psi^* \psi \, dx 
\geq 0.
\ee
\end{subequations}
\end{proposition}
 
\begin{proof}
Expanding out the bracket in $E^+(s, \psi)$ gives 
\begin{align*}
E^+(s, \psi) 
& :=
\int_{\Hcal_s} \Big( \psi^* \psi - 2 \frac{x_i}{t} \psi^* \gamma^0 \gamma^i \psi  + \psi^* \frac{x^i x^k}{t^2} \gamma^0 \gamma^j \gamma^0 \gamma^l \psi \delta_{ij} \delta_{kl} \Big) \, dx
\\
& = \int_{\Hcal_s} \Big( \psi^* \psi - 2 \frac{x_i}{t} \psi^* \gamma^0 \gamma^i \psi  - \psi^* \frac{x_i x_j}{t^2} \gamma^{(i} \gamma^{j)} \psi \Big) \, dx
\\
& = \int_{\Hcal_s} \Big( \psi^* \psi - 2 \frac{x_i}{t} \psi^* \gamma^0 \gamma^i \psi + \psi^* \psi \frac{x^i x^j}{t^2} \delta_{ij} \Big) \, dx
\\
& = \int_{\Hcal_s} \Big( \psi^* \psi - 2 \frac{x_i}{t} \psi^* \gamma^0 \gamma^i \psi + \psi^* \psi \frac{t^2-s^2}{t^2} \Big) \, dx
 = 2 E^\Hcal (s, \psi) - \int_{\Hcal_s} \frac{s^2}{t^2} \psi^* \psi \, dx, 
\end{align*}
and re-arranging gives \eqref{eq:energy-identity}.
\end{proof}

Using the above positivity property we can now establish the following energy inequality.

\begin{proposition}[Hyperboloidal energy estimate for Dirac spinor]
\label{prop-hyperboloidal-energy-inequality-Dirac}
For the massive Dirac spinor described by \eqref{Dirac1} the following estimate holds
\bel{eq:1st-EE} 
E^\Hcal (s, \psi)^{1/2} 
\leq E^\Hcal (s_0, \psi)^{1/2} + \int_{s_0}^s \| F \|_{L^2_f(\Hcal_{\bar{s}})} \,d\bar{s}.
\ee
\end{proposition}

\begin{proof}
Differentiating \eqref{integrated-energy-equality} with respect to $s$ yields
$$ 
\aligned
E^\Hcal(\bar{s}, \psi)^{1/2} \frac{d}{d \bar{s}} E^\Hcal(\bar{s}, \psi)^{1/2} 
& \leq 
{1\over 2}\int_{\Hcal_{\bar{s}}} (\bar{s}/t) \big(| F^* \gamma^0 \psi | + | \psi^* \gamma^0 F| \big) \,dx
\\
& \leq 
\| (\bar{s}/t) \psi \|_{L^2_f(\Hcal_{\bar{s}})} \| F \|_{L^2_f(\Hcal_{\bar{s}})}.
\endaligned
$$
Recalling in Proposition \ref{eq:positive-energy-1storder}, we have $\| (\bar{s}/t) \psi \|_{L^2_f(\Hcal_{\bar{s}})} \leq E^\Hcal(\bar{s}, \psi)^{1/2}$. Thus we have
$$ 
\frac{d}{d \bar{s}} E^\Hcal(\bar{s}, \psi)^{1/2} 
\leq \| F \|_{L^2_f(\Hcal_{\bar{s}})}, 
$$
and the conclusion follows by integrating over $[s_0, s]$. 
\end{proof}


\subsection{Hyperboloidal energy based on the second-order formulation} 
\label{subsec:2ndorder-dirac-energy}

Finally in this section we will convert the Dirac equation into a second-order PDE and define associated energy functionals. Apply the first-order Dirac operator $-i \gamma^\nu \partial_\nu$ to the Dirac equation \eqref{Dirac1} and use the identity \eqref{gamma-identities} to obtain
\be 
\eta^{\mu \nu} \del_\mu \del_\nu \psi + M(-i \gamma^\nu \del_\nu \psi) 
= - i \gamma^\nu \del_\nu F.
\ee
Substituting the PDE into the bracketed term gives the following second-order PDE
\bel{eq:psi-second-order}
\Box \psi - M^2 \psi 
= -MF - i \gamma^\nu \del_\nu F.
\ee
This provides us with another approach for deriving an energy estimate for the Dirac equation. 
We now check the hyperboloidal energy coming from \eqref{eq:psi-second-order}. 

\begin{lemma}[Second-order hyperboloidal energy estimate for the Dirac equation]
\label{lem:KG-energy-est}
For all solution $\psi \in \CC^4$ of 
\begin{subequations}
\bel{eq:psi-second-order-G}
\Box \psi - M^2 \psi 
= G,
\ee
one has 
\be 
\Ecal_M (s, \psi)^{1/2} 
\leq \Ecal_M(s_0, \psi)^{1/2} + \int_{s_0}^s \| G \|_{L^2(\Hcal_\tau)} d \tau,
\ee
where
\bel{energy-Dirac2}
\Ecal_M(s, \psi) 
:=
\int_{\Hcal_s} 
\Big( |\del_t \psi|^2 + \sum_i |\del_i \psi|^2 + M^2|\psi|^2 
+ 2 \sum_i \frac{x^i}{t} \mathrm{Re} \big[\del_t \psi \del_i \psi^* \big] \Big) dx.
\ee
\end{subequations}
\end{lemma}

\begin{proof}
The conjugate of \eqref{eq:psi-second-order-G} reads 
$
\Box \psi^* - M^2 \psi^* 
= G^*.
$
Using $-\del_t \psi^*$ and $-\del_t \psi$ as multipliers on \eqref{eq:psi-second-order-G} and its conjugate respectively we obtain
$$ 
\aligned
\del_t \psi^* \del^2_t \psi - \del_t \psi^* \sum_i \del_i^2 \psi + M^2 \del_t \psi^* \cdot \psi 
&= - \del_t \psi^* G,
\\
\del^2_t \psi^* \del_t \psi - \sum_i \del_i^2 \psi^* \del_t \psi + M^2 \del_t \psi \cdot \psi^* 
&= - G^* \del_t \psi.
\endaligned
$$
Adding these two equations together gives
$$ 
\aligned
&\quad 
\del_t \big( \del_t \psi^* \del_t \psi + \del^i \psi^* \del_i \psi + M^2\psi^* \psi \big) 
- \sum_i \del_i \big( \del_t \psi^* \del_i \psi + \del_t \psi \del_i \psi^* \big) 
\\
&= - \big( \del_t \psi^* G + G^* \del_t \psi) = - 2 \text{Re} \big[G^* \del_t \psi \big].
\endaligned
$$
Integrating this equation in the region $\mathcal{K}_{[s_0, s]}$, we have 
$$ 
\aligned
& \int_{\Hcal_s} 
\Big( |\del_t \psi|^2 + \sum_i |\del_i \psi|^2 + M^2|\psi|^2, - (\del_t \psi^* \del_i \psi + \del_i \psi^* \del_t \psi) \Big) \cdot n \, d \sigma
\\
& - \int_{\Hcal_{s_0}} 
\Big(|\del_t \psi|^2 + \sum_i |\del_i \psi|^2 + M^2|\psi|^2, - (\del_t \psi^* \del_i \psi + \del_i \psi^* \del_t \psi) \Big) \cdot n \, d \sigma
\\
&  =  - 2 \int_{\mathcal{K}_{[s_0, s]}}  \text{Re} \big[ G^* \del_t \psi \big] dt dx.
\endaligned
$$
Using the explicit form of $n$ and $d\sigma$  given in \eqref{eq:n-dsigma} and noting that $2 \text{Re}[\del_i \psi^* \del_t \psi] = \del_t \psi^* \del_i \psi + \del_i \psi^* \del_t \psi$ we find
\bel{eq:second-order-energy-equality}
\Ecal_M(s, \psi) - \Ecal_M(s_0,\psi)
= - 2 \int_{s_0}^s \int_{\Hcal_t} \text{Re} \big[ G^* \del_t \psi \big] dt dx.
\ee
We next use the change of variable formula $\del_i = \underline{\del}_i - (x^i/t)\del_t$ to rewrite the energy term:
$$ 
\aligned
& \int_{\Hcal_s} 
\Big( |\del_t \psi|^2 + \sum_i |\del_i \psi|^2 + M^2|\psi|^2 + \frac{x^i}{t} \big( \del_t \psi^* \del_i \psi + \del_i \psi^* \del_t \psi \big) \Big) dx
\\
& = \int_{\Hcal_s}
\Big( |(s/t) \del_t \psi |^2 + \sum_i | \underline{\del}_i \psi |^2 + M^2 |\psi|^2 \Big) dx.
\endaligned
$$
We can now estimate the nonlinearity on the RHS using the change of variables $\tau = (t^2-r^2)^{1/2}$ and $dt dx = (\tau/t) d\tau dx$. In particular we have
$$ 
\aligned
-2 \int_{\Hcal_\tau} \text{Re}\big[ G^* \del_t \psi \big] (\tau/t) \, dx
& \leq
2 \| (\tau/t) \del_t \psi \|_{L^2(\Hcal_\tau)} \| G \|_{L^2(\Hcal_\tau)}
\\
& \leq 2 \Ecal_M(\tau, \psi)^{1/2}  \| G \|_{L^2(\Hcal_\tau)}.
\endaligned
$$
Thus by differentiating \eqref{eq:second-order-energy-equality} and using the above we have
$$ 
\frac{d}{d\tau} \Ecal_M(\tau, \psi)^{1/2} 
\leq \| G \|_{L^2(\Hcal_\tau)}.
$$
Integrating this expression over $[s_0, s]$ gives the desired result. 
\end{proof}

We end with a short remark here on positive and negative mass spinors.
In the original first-order Dirac equation \eqref{Dirac1} the mass $M$ was defined as a real parameter with no sign restriction. Furthermore the mass did not appear in the hyperboloidal energy $E^\Hcal$ defined in \eqref{energy-Dirac1}. This implies that spinors with equal masses but opposite signs ($\pm M$) would still obey the same energy estimates of Proposition \ref{prop-hyperboloidal-energy-inequality-Dirac}. 

This is consistent with the second-order equation \eqref{eq:psi-second-order} for the Dirac field. In this equation the mass appears squared, so spinors with equal masses, but of opposite signs, obey the \textit{same} second order equation. Moreover the mass $M^2$ appears in the second-order hyperboloidal energy expression $\Ecal_M$ in \eqref{energy-Dirac2}. 

Thus either the mass $M$ should not appear in $E^\Hcal$, as we have found, or if the mass $M$ were to appear in $E^\Hcal$, it would necessarily need to be invariant under a sign change, so as to agree with the second-order energy estimates involving $\Ecal_M$. 


\section{Additional properties of Dirac spinors on hyperboloids} 
\label{sec:choleksy-weyl}

\subsection{Hyperboloidal energy based on a Cholesky decomposition} 
\label{subsec:choleksy}

Our first task is to obtain a hyperboloidal energy for the Dirac field $\psi$ expressed in terms of a product of complex vectors $z(\psi)^* z(\psi)$. Such an expression is then easily seen to be positive semi-definite which, clearly, is in contrast to the form given in \eqref{energy-Dirac1} and Proposition \ref{eq:positive-energy-1storder}. 

Recall the standard Cholesky decomposition: any Hermitian, positive-definite matrix $A$ can be decomposed in a unique way as 
\be 
A = P^* P, 
\ee
where $P$ is a lower triangular matrix with real and positive diagonal entries. In particular if $A$ is positive semi-definite then the decomposition exists however one loses uniqueness and the diagonal entries of $P$ may be zero. 


We now prove the following result. 

\begin{proposition}[Hyperboloidal energy for the Dirac equation. II] 
\label{appendix:hyperboloidal-energy-P-matrix}
There exists a lower triangular matrix $P$ with real and positive diagonal entries such that
\be 
E^\Hcal(s, \psi) 
= \int_{\Hcal_s} (P \psi) ^* (P \psi ) dx, 
\ee
and specifically 
\bel{eq:explicitP}
P = 
\begin{pmatrix}
	s/t & 0 & 0 & 0 \\
	0 & s/t & 0 & 0 \\
	x^3/t & x^1/t -i x^2/t & 1 & 0 \\
	x^1/t +i x^2/t & - x^3/t & 0 & 1
\end{pmatrix}, 
\ee
which can also be expressed as
\be 
P 
= \frac{(s/t)+1}{2} I_4 + \frac{(s/t)-1}{2} \gamma^0 + \delta_{ij} \frac{x^i}{t}\gamma^0 \gamma^j.
\ee
\end{proposition}

The above expression is quite natural and resembles what is known for the wave equation: the factor $x^i/t$ comes from Stokes' theorem applied to hyperboloids and we cannot expect to fully control the standard $L^2$ norm, namely
$
\int_{\Hcal_s} \psi^* \psi \, dx. 
$

\begin{proof} {\bf Step 1. Existence of the decomposition.} Before we proceed with the derivation of the identity, we present an argument showing that such a decomposition exists by proving positive semi-definiteness. For simplicity of notation, let $N_i := x^i/t$.
The integrand of $E^\Hcal(s, \psi)$ can be written as $\psi^* A \psi $ where $A:= I_4 + N_j \gamma^0 \gamma^j$. Here the spatial indices are contracted with $\delta_{ij}$, so that $N_j \gamma^j = \delta_{ij} N^i \gamma^j$.   
Note $A$ is hermitian since $A^* = I_4 + N_j (\gamma^0 \gamma^j)^* = A$.
Also
$$ 
(N_j \gamma^0 \gamma^j)(N_k \gamma^0 \gamma^k)
= - N_j N_k \gamma^0 \gamma^0 \gamma^j \gamma^k 
= - N_j N_k \gamma^{(j} \gamma^{k)} 
= N_j N^j I_4.
$$
Then for all $z \in \CC^4$ we have
$$
\aligned
0 
\leq \big( A z \big)^* \big( A z \big)
& = (1+N_j N^j) z^* I_4 z  + 2 z^* N_j \gamma^0 \gamma^j z 
\\
& \leq 2 \left( z^* I_4 z + z^* N_j \gamma^0 \gamma^j z \right) 
\\
& = 2 z^* A z.
\endaligned
$$
We used that $N_j N^j  = (r/t)^2 \leq 1$ which holds in the light-cone $\mathcal{K}$. Thus $A$ is positive semi-definite.

\vskip.3cm

{\bf Step 2. Computing the matrix $P$.} 
With respect to the Dirac representation \eqref{Dirac-representation} we have
$$ 
\aligned
A 
& = I_4 + N_j 
	\begin{pmatrix} I_2 & 0_2 \\ 0_2 & - I_2 \end{pmatrix} 
	\begin{pmatrix} 0_2 & \sigma^j \\ - \sigma^j & 0_2 \end{pmatrix}
\\
& = \begin{pmatrix} I_2 & 0_2 \\ 0_2 & I_2 \end{pmatrix} 
	+ N_j \begin{pmatrix} 0_2 & \sigma^j \\ \sigma^j & 0_2 \end{pmatrix}.
\endaligned
$$ 
Here $I_2$ and $0_2$ represent the $2 \times 2$ identity and zero matrices respectively. 
Calculate the second term above using the Pauli matrices:
$$ 
N_j \begin{pmatrix} 0 & \sigma^j \\ \sigma^j & 0 \end{pmatrix} 
= N_1 \begin{pmatrix} 0 & 1 \\ 1 & 0 \end{pmatrix} 
	+ N_2 \begin{pmatrix} 0 & -i \\ i & 0 \end{pmatrix}
	+ N_3 \begin{pmatrix} 1 & 0 \\ 0 & -1 \end{pmatrix}
= \begin{pmatrix} N_3 & N_1 - i N_2 \\ N_1 + i N_2 & - N_3 \end{pmatrix}.
$$
Define $\omega := N_1 + i N_2$ and recall $N_i \in \RR$. Thus we have
$$ 
A = 
\begin{pmatrix} I_2 
	& \begin{matrix} N_3 & \bar{\omega} \\ \omega & - N_3 \end{matrix} \\
	\begin{matrix} N_3 & \bar{\omega} \\ \omega & - N_3 \end{matrix} 
	& I_2
\end{pmatrix}.
$$
Consider now $2\times 2$ complex matrices $B, C, D$ such that 
$$ 
\begin{pmatrix} B & 0 \\ C & D \end{pmatrix} ^ *
\begin{pmatrix} B & 0 \\ C & D \end{pmatrix}
= 
\begin{pmatrix} I_2 
	& \begin{matrix} N_3 & \bar{\omega} \\ \omega & - N_3 \end{matrix} \\
	\begin{matrix} N_3 & \bar{\omega} \\ \omega & - N_3 \end{matrix} 
	& I_2
\end{pmatrix}.
$$
This implies the following identities
$$ 
\aligned
D^* D & = I_2, \\
C^* D & = D^* C = 
	\begin{pmatrix} N_3 & \bar{\omega} \\ \omega & - N_3 \end{pmatrix}, \\
B^* B + C^* C & = I_2.
\endaligned
$$
If we let $D = I_2$ and $C = \begin{pmatrix} N_3 & \bar{\omega} \\ \omega & - N_3 \end{pmatrix}$ then we must solve
$$ 
B^* B  = I_2 - 
	\begin{pmatrix} N_3 & \bar{\omega} \\ \omega & - N_3 \end{pmatrix}^*
	\begin{pmatrix} N_3 & \bar{\omega} \\ \omega & - N_3 \end{pmatrix}	
 = \lambda \begin{pmatrix} 
		1 & 0 \\
		0 & 1
	\end{pmatrix},
$$
where $\lambda := 1 - (N_3^2 + \bar{\omega} \omega )  = 1 - (N_1^2 + N_2 ^2 + N_3 ^2 )$. Indeed $\lambda = 1 - (r/t)^2 = (s/t)^2 \geq 0$ so we can take $B = \sqrt{\lambda}I_2 = (s/t) I_2$. 
\end{proof}


\subsection{Hyperboloidal energy based on the Weyl spinor representation} 
\label{subsec:Weyl}

Yet one more approach in deriving energy estimates is obtained by expressing the Dirac spinors in terms of Weyl spinors and then studying the energy of Weyl spinors \eqref{eq:energy-Weyl} instead. This  provides another convenient way to study Dirac equations.
Decompose the spinor $\psi$ and source term $F$ as 
\bel{Weyl-spinors}
\psi 
= \begin{pmatrix} u+v \\ u-v \end{pmatrix}, 
\quad 
F 
= \begin{pmatrix} F_+ + F_- \\ F_+ - F_- \end{pmatrix}, 
\ee
where $u,v:\RR^{1+3} \to \CC^2$ are Weyl spinors and $F_\pm \in \CC^2$. 
Defining $\del_\pm := \del_0 \pm \sigma^i \del_i$ the PDE \eqref{Dirac1} can be shown to be equivalent to
\bel{Weyl1} 
\aligned
\del_- v + iMu 
&= iF_+,
\\
\del_+ u + iMv 
&= i F_-.
\endaligned
\ee
A Dirac-Klein-Gordon system with respect to such a Weyl spinor decomposition has been studied, albeit in the low-regularity setting, by Bournaveas \cite{Bournaveas}. 
Following a similar approach to Section \ref{subsec:dirac-energies} we find an analogous hyperboloidal Weyl spinor energy:
\bel{eq:energy-Weyl} 
E^\sigma_\pm(s, u) 
:= \int_{\Hcal_s} \big( u^*u \pm \frac{x_j}{t} u^* \sigma^j u \big) \, dx.
\ee
Similar to Propositions \ref{eq:positive-energy-1storder} and \ref{prop-hyperboloidal-energy-inequality-Dirac} we can prove positivity and an energy estimate for $E^\sigma_\pm$. 

\begin{proposition}
For a $\CC^2$--valued function $w$ the following holds:
\be 
E^\sigma_\pm(s, w) 
\geq \frac{1}{2} \int_{\Hcal_s} \frac{s^2}{t^2} w^* w \,dx \geq 0.
\ee
Furthermore for solutions $u, v$ to \eqref{Weyl1} we have
\be 
\big(E^\sigma _+ (s, u) + E^\sigma _- (s, v) \big)^{1/2} 
\leq \big(E^\sigma _+ (s_0, u) + E^\sigma _- (s_0, v) \big)^{1/2} 
+ \int_{s_0}^s \| F_+ \|_{L^2(\Hcal_{\bar{s}})} + \| F_- \|_{L^2(\Hcal_{\bar{s}})} d \bar{s}.
\ee
\end{proposition}

\begin{proof} {\bf Step 1.} Using the Dirac representation \eqref{Dirac-representation} and the decomposition \eqref{Weyl-spinors}, the PDE \eqref{Dirac1} becomes
\be 
	\begin{pmatrix} 
		I_2 & 0 \\ 0 & - I_2 
	\end{pmatrix} 
\del_0 \begin{pmatrix} 
		u+v \\ u-v 
		\end{pmatrix}
+ \begin{pmatrix} 
	0 & \sigma^j \\ -\sigma^j & 0 
	\end{pmatrix} 
\del_j \begin{pmatrix} 
		u+v \\ u-v 
		\end{pmatrix} 
+ iM \begin{pmatrix} 
	u+v \\ u-v 
	\end{pmatrix}
 = i \begin{pmatrix} 
 	F_+ + F_- \\ F_+ - F_- 
 	\end{pmatrix}.
\ee
Defining $\del_\pm := \del_0 \pm \sigma^i \del_i$ this becomes
\bel{A-Weyl-pde-matrix-form} 
\begin{pmatrix} 
	\del_+ u + \del_- v \\ \del_- v - \del_+ u 
\end{pmatrix}
+ iM \begin{pmatrix} 
	u+v \\ u-v 
	\end{pmatrix}
 = i \begin{pmatrix} 
 	F_+ + F_- \\ F_+ - F_- 
 	\end{pmatrix}.
\ee
Adding and subtracting the two rows above gives the following
$$ 
\aligned
\del_- v + iMu 
&= iF_+,
\\
\del_+ u + iMv 
&= i F_-.
\endaligned
$$
Following a similar approach to deriving \eqref{del-Dirac}, we multiply the first and second equation by $v^*$ and $u^*$ respectively. 
$$ 
\aligned
u^* \del_0 u + u^* \sigma^j \del_j u + iM u^* v 
&= i u^* F_-, 
\\
v^* \del_0 v - v^* \sigma^j \del_j v + iM u^* v 
&= -i v^* F_+.
\endaligned
$$
One then adds these equations to their conjugate to obtain the following:
$$
\aligned
\del_0(u^*u) + \del_j(u^*\sigma^j u) + iM(u^*v-v^*u) 
&= i u^* F_- - i F_-^* u,
\\
\del_0(v^*v) - \del_j(v^*\sigma^j v) + iM(v^*u - u^*v) 
&= iv^* F_+ - i F_+^* v.
\endaligned
$$
Note the mass terms appear above. However if add these equations together we obtain
\bel{A-del-Weyl}
\del_0(u^*u + v^*v) + \del_j \big(u^*\sigma^j u - v^*\sigma^j v \big) 
= i u^* F_- - i F_-^* u + iv^* F_+ - i F_+^* v,
\ee
which does not contain a term involving $M$. This equation is the analogous Weyl spinor version of \eqref{del-Dirac}. Clearly integrating \eqref{A-del-Weyl} over $\mathcal{K}_{[s_0, s]}$ gives the energy functional $E^\sigma_\pm(s, u) $ defined in \eqref{eq:energy-Weyl}.

{\bf Step 2.} Next we establish that for a $\CC^2$--valued function $w$ the following holds:
$$
E^\sigma_\pm(s, w) 
\geq \frac{1}{2} \int_{\Hcal_s} \frac{s^2}{t^2} w^* w \,dx \geq 0.
$$
The idea is in the spirit of Proposition \ref{eq:positive-energy-1storder}. 
Observe that the sigma matrices are Hermitian and satisfy the following anti-commutator relation: $\{\sigma^i, \sigma^j\} = 2 \delta^{ij} I_2$. Then we have
$$
\aligned
\int_{\Hcal_s} \Big( w \pm \frac{x_j}{t} \sigma^j w\Big)^*\Big(w \pm \frac{x_j}{t} \sigma^j w\Big) \,dx 
& = \int_{\Hcal_s} \Big( w^*w \pm 2 \frac{x_j}{t} w^* \sigma^j w + \frac{x_j x_k}{t^2} w^* \sigma^{(j} \sigma^{k)} w \Big) \,dx 
\\
&= \int_{\Hcal_s} \Big( w^*w (1+(r/t)^2) \pm 2 \frac{x_j}{t} w^* \sigma^j w \Big) \,dx 
\\
& = 2 E^\sigma_\pm(s, w) - \int_{\Hcal_s} \frac{s^2}{t^2} w^* w \,dx.
\endaligned
$$
Thus we have 
$$
 E^\sigma_\pm(s, w) 
= \frac{1}{2} \int_{\Hcal_s} \Big( w \pm \frac{x_j}{t} \sigma^j w\Big)^*\Big(w \pm \frac{x_j}{t} \sigma^j w\Big) \,dx 
+ \frac{1}{2} \int_{\Hcal_s} \frac{s^2}{t^2} w^* w \,dx 
\geq 0.
$$

\

\noindent{\bf Step 3.} Next, let us show that the following hyperboloidal energy inequality holds for the Weyl spinor equation \eqref{Weyl1}
$$ 
\big(E^\sigma _+ (s, u) + E^\sigma _- (s, v)\big)^{1/2} 
\leq \big(E^\sigma _+ (s_0, u) + E^\sigma _- (s_0, v) \big)^{1/2} 
+ \int_{s_0}^s \| F_+ \|_{L^2(\Hcal_{\bar{s}})} + \| F_- \|_{L^2(\Hcal_{\bar{s}})} d \bar{s}.
$$
Namely, integrating \eqref{A-del-Weyl} over $\mathcal{K}_{[s_0, s]}$ we obtain
$$
\aligned
&\quad 
E^\sigma _+ (s, u) + E^\sigma _- (s, v) 
\\
& = E^\sigma _+ (s_0, u) + E^\sigma _- (s_0, v)
    + \int_{s_0}^s d \bar{s} \int_{\Hcal_{\bar{s}}} (\bar{s}/t) (i u^* F_- - i F_-^* u + iv^* F_+ - i F_+^* v) \,dx.
\endaligned
$$
Differentiating in $s$ and noting that $\| (s/t) u, (s/t)v \|_{L^2(\Hcal_s)} \leq \|u, v\|_{L^2(\Hcal_s)} \leq ( E^\sigma _+ (s, u) + E^\sigma _- (s, v) )^{1/2} $ gives
$$
\frac{d}{ds}\Big( E^\sigma _+ (s, u) + E^\sigma _- (s, v)\Big)^{1/2} 
\leq \| F_+ \|_{L^2(\Hcal_\tau)} + \| F_- \|_{L^2(\Hcal_\tau)}. 
$$
\end{proof}


\subsection{Sobolev-type estimates for a Dirac spinor} 
\label{sec:Sobolev-est-Dirac}
In this section we will obtain novel decay estimates for the Dirac spinor. First we recall the following standard Sobolev estimate on hyperboloids.

\begin{proposition}[Sobolev-type inequality on hyperboloids]
\label{prop:Sobolev-type-inequality3}
Suppose $\psi= \psi(t, x)$ is a sufficiently smooth spinor field supported in the region $\{(t, x): |x|< t - 1\}$, then it holds for $s\geq 2$
\bel{eq:Sobolev3}
\sup_{\Hcal_s} \big| t^{3/2} \psi(t, x) \big| 
\lesssim \sum_{| J |\leq 2} \big\| L^J \psi \big\|_{L^2_f(\Hcal_s)},
\ee
where the summation is over  Lorentz boosts $L$.
\end{proposition}

This estimate will be stated again in Section \ref{sec:bootstrap} and proved for general functions, see Proposition \ref{prop:Sobolev-type-inequality2}. 
In this section, we will adapt this Sobolev-type inequality \eqref{eq:Sobolev3} to include boosts which commute with the Dirac operator $i \gamma^\nu \del_\nu$.
Following Bachelot \cite{Bachelot} we introduce modified boosts $\Lh_a$  that differ from $L_a$ by a constant matrix
\bel{eq:modified-lorentz-boosts}
\Lh_a 
:= L_a + \frac{1}{2} \gamma_0 \gamma_a. 
\ee
It then holds that $[\Lh_a, i \gamma^\nu \del_\nu] = 0$. The following result is a simple extension of Proposition \ref{prop:Sobolev-type-inequality3}.

\begin{corollary}
\label{cor:Sobolev-type-inequality3}
Suppose $\psi= \psi(t, x)$ is a sufficiently smooth spinor field supported in the region $\{(t, x): |x|< t - 1\}$, then it holds for $s\geq 2$
\bel{eq:Sobolev2-mm}
\sup_{\Hcal_s} \big| t^{3/2} \psi(t, x) \big| 
\lesssim \sum_{| J |\leq 2} \big\| \Lh^J \psi \big\|_{L^2_f(\Hcal_s)}, 
\ee
where $\Lh$ denotes a modified Lorentz boost.
\end{corollary}

We can combine this Sobolev estimate \eqref{eq:Sobolev2-mm} with an appropriate energy bound in a standard way to obtain the following `weak'  decay estimate for the Dirac field. 

\begin{proposition}
Assume $E^\Hcal (s, \Lh^J \psi)^{1/2} \leq \eps_1 s^{\delta_1}$ for $|J| \leq 2$ with $\eps_1 > 0, \delta_1 \geq 0$ some parameters, then 
\bel{eq:decay-part1}
\big| s t^{-1/2} \psi \big|
\lesssim \eps_1 s^{\delta_1}.
\ee
\end{proposition}
\begin{proof}
It is not hard to show
$$
\Big\| \Lh^J \big((s/ t) \psi\big) \Big\|_{L^2_f} \, dx
\lesssim \eps_1 s^{\delta_1},
\qquad
|J| \leq 2.
$$
Then together with the Sobolev inequality on hyperboloids \eqref{eq:Sobolev3}, it gives the desired pointwise decay estimate \eqref{eq:decay-part1}.
\end{proof}

Since $t < s^2$ within the cone
$\mathcal{K}$, the decay rate \eqref{eq:decay-part1} is not as good as that of standard Klein-Gordon fields, which leads us to seek help from the other component of the functional $E^\Hcal (s, \psi)$.
In Proposition \ref{eq:positive-energy-1storder}, we have shown the decomposition identity \eqref{eq:energy-identity} for the hyperboloidal energy functional $E^\Hcal (s, \psi)$ of the Dirac equation, i.e.
$$
\aligned
E^\Hcal (s, \psi) 
&= \frac{1}{2} E^+(s, \psi) + {1 \over 2} \int_{\Hcal_s} {s^2 \over t^2} \psi^* \psi \, dx,
\\
E^+(s, \psi) 
&:= 
\int_{\Hcal_s} 
\Big(\psi + \frac{x_i}{t} \gamma^0 \gamma^i \psi \Big)^* 
\Big(\psi + \frac{x_j}{t} \gamma^0 \gamma^j \psi \Big) \, dx
\\
&=
\int_{\Hcal_s} 
\Big(\gamma^0 \psi + {x_i \over t} \gamma^i \psi \Big)^* 
\Big(\gamma^0 \psi + {x_j \over t} \gamma^j \psi \Big) \, dx.
\endaligned
$$
We notice that there is no information about the mass $M$ in the Dirac energy $E^\Hcal$. Nonetheless by closely studying each component of the functional $E^\Hcal (s, \psi)$, we obtain the following estimate. 

\begin{proposition}\label{prop:energy-decop}
Consider a solution to the Dirac equation 
\be
-i \gamma^\mu \del_\mu \psi + M \psi 
= F,
\ee 
where $|F| \lesssim \eps_1 t^{-3/2}$, and assume $E^\Hcal (s, \del^I \Lh^J \psi)^{1/2} \leq \eps_1 s^{\delta_1}$ for $|I| + |J| \leq 3$ with $\eps_1 > 0, \delta_1 \geq 0$ some parameters, then 
\be 
\big| M t^{3/2} \psi \big|
\lesssim \eps_1 s^{ \delta_1}.
\ee
\end{proposition}
\begin{proof}
$Step~1.$
We first express the Dirac equation \eqref{Dirac1} in the semi-hyperboloidal frame to get
$$
-i \Big( \gamma^0 - \gamma^a {x_a \over t} \Big) \del_t \psi
- i \gamma^a \underline{\del}_a \psi
+ M \psi
= F,
$$
which gives
$$
M \psi
=
F + i \Big( \gamma^0 - \gamma^a {x_a \over t} \Big) \del_t \psi
+ i \gamma^a \underline{\del}_a \psi.
$$

It is easy to see that the first and the third terms are bounded by $\eps_1 t^{-3/2 + \delta_1}$, hence it is enough to show 
$$
\Big| \Big( \gamma^0 - \gamma^a {x_a \over t} \Big) \del_t \psi \Big|
\lesssim \eps_1 t^{-3/2 + \delta_1}.
$$
Then recalling the Sobolev inequality on hyperboloids \eqref{eq:Sobolev2}, it suffices to show
\be
\Big\| \Lh^J \Big( \big( \gamma^0 - \gamma^a (x_a / t) \big) \del_t \psi \Big) \Big\|
\lesssim \eps_1 s^{\delta_1},
\qquad
|J| \leq 2.
\ee

$Step~2.$
We note that Proposition \ref{eq:positive-energy-1storder} implies 
$$
\Big\| \big( \gamma^0 - \gamma^a (x_a / t) \big) \Lh^J \del_t \psi \Big\|
\lesssim \eps_1 s^{\delta_1},
\qquad
|J| \leq 2,
$$
which suggests us to study the commutator
$$
\big[ \Lh_a, \gamma^0 - \gamma^a (x_a / t) \big].
$$

A computation gives
$$
\big[ \Lh_a, \gamma^0 - \gamma^a (x_a / t) \big]
=
- {1\over t} \gamma^a - {x_a \over t} \big( \gamma^0 - \gamma^a (x_a / t) \big),
$$
which only contains good terms. 
An induction shows $\big[ \Lh^J, \gamma^0 - \gamma^a (x_a / t) \big]$ with $|J| \leq 2$ only contains good terms.
Hence the proof is complete.
\end{proof}

 
\section{Nonlinear stability of the ground state for the Dirac-Proca model} 
\label{section:DiracProca}

\subsection{The Dirac-Proca model as a PDE system} 
\label{subsec:dp-pde}

Using the tools of Section \ref{sec:hyper-energies}, we will now discuss the Dirac-Proca model, the gauge condition and choice of initial data. This leads us to the second main stability Theorem \ref{theo-oursecondone} below, which can be proved using the methods of Sections \ref{sec:nonlinearities} and \ref{sec:bootstrap}.  Without fixing a gauge, the field equations for the Dirac-Proca model with unknowns $A^\mu$ and $\psi$ read 
\bel{eq:101ng} 
\aligned
\Box A^\nu - m^2 A^\nu  + \del^\nu (\del_\mu A^\mu)
& = 
	-\psi^* \gamma^0 \gamma^\nu (P_L \psi),
\\
-i \gamma^\mu \del_\mu \psi  
&= 
	- \gamma^\mu A_\mu (P_L \psi). 
\endaligned 
\ee
Recall $P_L = {1\over 2} (I_4 - \gamma^5)$ was defined in \eqref{eq:proj-ops}. 
Here  $\psi: \RR^{3+1} \to \CC^4$ represents a \textit{massless} Dirac fermion with spin $1/2$ and $A^\mu: \RR^{3+1} \to \RR$ represents a massive boson (the Proca field) of mass $m^2$ with spin $1$. 
As discussed in \cite{Tsutsumi}, \eqref{eq:101ng} is equivalent to the following system
\bel{eq:102g} 
\aligned
\Box A^\nu - m^2 A^\nu  
& = 
	-\psi^* \gamma^0 \gamma^\nu (P_L \psi),
\\
-i \gamma^\mu \del_\mu \psi  
&= 
	- \gamma^\mu A_\mu (P_L \psi),
\\
\del_\mu A^\mu
& = 0.
\endaligned 
\ee

It can easily be shown that $\del_\mu A^\mu$ satisfies a homogeneous Klein-Gordon equation. Thus we may treat $\del_\mu A^\mu = 0$ as a constraint provided we specify the initial data set $\big(A_0^\nu, A_1^\nu, \psi_0 \big)$ at some time $t_0>0$: 
\bel{eq:101-ID2} 
\aligned 
A^\nu (t_0,\cdot) 
= a^\nu,
\qquad
\del_t A^\nu (t_0, \cdot) = b^\nu,
\qquad
\psi(t_0,\cdot) 
= \psi_0.
\endaligned
\ee
to satisfy the following two `Lorenz compatibility' conditions
\bel{eq:locc}
\aligned
b^0 + \del_j a^j 
	& = 0,
\\ 
\Delta a^0 - m^2 a^0 
& = 
	-\psi_0^* P_L \psi_0 - \del_j b^j.
\endaligned
\ee
For more generality, see also \cite{Bachelot}, we consider \eqref{eq:101ng} with an artificial mass $M$ added to the spinor:
\bel{eq:intro-dirac-proca2} 
\aligned
\Box A^\nu - m^2 A^\nu  
& = 
	-\psi^* \gamma^0 \gamma^\nu (P_L \psi),
\\
-i \gamma^\mu \del_\mu \psi   + M \psi
&= 
	- \gamma^\mu A_\mu (P_L \psi).
\endaligned 
\ee
The mass parameters $M \geq 0$ and $m > 0$ are constants, and we will study both cases $M = 0$ and $M > 0$.
Again the initial data will be taken to satisfy \eqref{eq:locc}. This elliptic-type system of two equations for nine scalar functions admits non-trivial compactly supported solutions. For example one may choose: $A_0^j \in C_c(\RR^{3+1})$ for $j=1,2,3$, $\psi_0 \in C_c(\RR^{3+1})$ such that each component of $\psi_0$ is the same, and all remaining initial data is trivial. 
 

\subsection{Main result for the Dirac-Proca model} 

We now state our result for the Dirac-Proca model. The proof of the theorem below will clearly follow from the  the proof we will develop for our main result (Theorem~\ref{thm:full-U(1)}), and so we omit it.  

For the model under consideration in this section, the {\sl ground state} of the theory is simply given by 
\be
A 
\equiv 0, 
\qquad 
\psi 
\equiv 0. 
\ee

\begin{theorem}
\label{theo-oursecondone}
Consider the Dirac-Proca system \eqref{eq:intro-dirac-proca}  with $M \geq 0$, $m > 0$, and let $N$ be a sufficiently large integer. There exists $\eps_0 > 0$, which is independent of $M$, such that for all $\eps \in (0, \eps_0)$ and all compactly supported, Lorenz compatible initial data in the sense of \eqref{eq:lorenzc}, satisfying the smallness condition
\be 
\| a^\nu, \psi_0\|_{H^N} + \|b^\nu \|_{H^{N-1}} 
\lesssim \eps,
\ee
the initial value problem of \eqref{eq:intro-dirac-proca} admits a unique global-in-time solution $(\psi, A^\nu)$.
Furthermore, the following decay result holds
\be 
|A| 
\lesssim 
\eps t^{-3/2},
\qquad
|\psi | 
\lesssim 
\eps \min \big( t^{-1}, M^{-1} t^{-3/2} \big).
\ee
\end{theorem}


\section{Nonlinear stability of the ground state for the $U(1)$ model} 
\label{section:U1model}

\subsection{The U(1) model as a PDE system}

We now treat the Higgs mechanism applied to a $U(1)$ abelian gauge field. This gives some exposure to the problems coming from the Higgs field that we will meet when dealing with the full GSW model in future work. From the Lagrangian of this theory; see \eqref{eq:lagrangian-u1} below, and 
in a suitably modified Lorenz gauge; see \eqref{eq:U1-constraint}  below, 
the field equations for this model 
read 
\bel{eq:u101}
\aligned
(\Box - m_q^2) A^\nu 
& = Q_{A^\nu},
\\
\Box \chi - m_q^2 \, {\phi_0 \over 2 v^2} \big(\phi_0^* \chi - \chi^* \phi_0 \big) - m_\lambda^2 \, {\phi_0 \over 2 v^2} \big(\phi_0^* \chi + \chi^* \phi_0 \big)
& = Q_{\chi},
\\
i \gamma^\mu \del_\mu \psi - m_g \psi 
&= Q_{\psi},
\endaligned
\ee
with quadratic or higher order terms given by
\bel{eq:u1-nl}
\aligned
Q_{A^\nu}
&:=
iq \big( \chi^* (\del^\nu \chi) - (\del^\nu \chi^*) \chi \big) + 2 q^2  A^\nu \big(\chi^* \phi_0 + \phi_0^* \chi + \chi^* \chi \big) + q \psi^* \gamma^0 \gamma^\nu \psi,
\\
Q_{\chi}
&:=
2iq A_\mu \del^\mu \chi + q^2 \chi \big(\phi_0^* \chi - \chi^* \phi_0 \big) + q^2 A^\mu A_\mu \big(\phi_0 + \chi \big)  
\\
&\quad
+ 2 \lambda \chi^* \chi \phi_0 + 2 \lambda \chi \big(\phi_0^* \chi + \chi^* \phi_0 + \chi^* \chi \big)
 - g \big(\phi_0 + \chi \big)\psi^* \gamma^0 \psi,
\\
Q_{\psi}
&:= 
g \big(\phi_0^* \chi + \chi^* \phi_0 + \chi^* \chi \big) \psi - q \gamma^\mu A_\mu \psi. 
\endaligned
\ee
Here $\chi := \phi - \phi_0:\RR^{1+3} \to \CC$ is the perturbation from a constant vacuum state $\phi_0$ satisfying $\phi_0^* \phi_0 = v^2$ and $\del_\mu \phi_0 = 0$. 
The field $\psi: \RR^{3+1} \to \CC^4$ represents a fermion of mass $m_g$ with spin $1/2$ and $A_\mu:\RR^{1+3} \to \RR$ represents a massive boson of mass $m_q^2$ with spin $1$. 
Furthermore the mass coefficients 
\be
m_q^2 
= 2 q^2 v^2, 
\quad
m_\lambda^2 
= 4 \lambda v^2, 
\quad
m_g 
= g v^2
\ee
depend themselves on given coupling constants denoted by $q, g, \lambda$, 
as well as the so-called ``vacuum expectation value'' of the Higgs field, denoted by $v$. 

The initial data set are denoted by
\bel{eq:u1-ID-mm}
\big( A^\nu, \chi, \psi \big) (t_0, \cdot) 
= \big( A^\nu_0, \chi_0, \psi_0 \big), 
\qquad
 \big( \del_t A^\nu, \del_t \chi \big) (t_0, \cdot) 
= 
\big( A^\nu_1, \chi_1 \big), 
\ee
which are said to be `Lorenz compatible' if they satisfy
\bel{eq:U1-constraint} 
\aligned
\del_a A_0^a 
&
	= - A^0_1  - i q \big( \phi_0^* \chi_0 - \chi_0^* \phi_0 \big),
\\
\Delta A^0_0 - m_q^2 A^0_0  
&=  - \del_i A_1^i - i q \big( \phi_0^* \chi_1 - \chi_1 ^* \phi_0 \big) + i q \big( \chi_0^* \chi_1 - \chi_1 ^* \chi_0 \big) 
\\
&\quad
+ 2 q^2 A^0_0 \big( \phi_0^* \chi_0 + \chi_0^* \phi_0 + \chi_0^* \chi_0 \big) + q \psi_0^* \psi_0.
\endaligned
\ee
The derivation of \eqref{eq:U1-constraint}  will follow from results of the following Section \ref{subsec:U1-inv}, in particular the gauge choice condition given in Lemma \ref{lem:u1-linearised-gauge}. 
Similar to the constraint equations \eqref{eq:locc} for the Dirac-Proca model, we also obtain in \eqref{eq:U1-constraint}  an elliptic-type system of only two equations for eleven functions. 
Clearly non-trivial solutions with compact support can be constructed. 

Our  main result was stated in the introduction (Theorem~\ref{thm:U1}) and the proof will be provided in Section~\ref{sec:bootstrap}.


\subsection{The abelian action and $U(1)$ invariance} \label{subsec:U1-inv}

The Lagrangian we consider is
\bel{eq:lagrangian-u1} 
\aligned
\Lcal 
= -{1 \over 4} F_{\mu \nu} F^{\mu \nu}
	- (D_\mu \phi)^* D ^\mu \phi - V(\phi^* \phi)
	- i \psi^* \gamma^0 \gamma^\mu D_\mu \psi
	+ g\phi^* \phi \psi^* \gamma^0 \psi,
	\endaligned
\ee
where we use the following definitions for the Higgs potential, gauge curvature and gauge covariant derivatives:
\bel{eq:83GG}
\aligned
V(\phi^* \phi) 
&:=\lambda ( \phi^* \phi - v^2)^2,
\\
F_{\mu \nu} 
&:= \del_\mu A_\nu - \del_\nu A_\mu,
\\
D_\mu \phi 
&:= (\del_\mu - i q A_\mu) \phi,
\\
D_\mu \psi 
&:= (\del_\mu - i q A_\mu) \psi.
\endaligned
\ee
Furthermore $\lambda, v, g, q$ are constants.
A calculation shows that the Euler-Lagrange equations for \eqref{eq:lagrangian-u1} are the following

\be 
\aligned
\del_\mu F^{\mu \nu} 
& = iq \phi^* (D^\nu \phi) - iq (D^\nu \phi)^* \phi + q \psi^* \gamma^0 \gamma^\nu \psi,
\\
D^\mu D_\mu \phi 
& = V' (\phi^* \phi) \phi - g \phi \psi^* \gamma^0 \psi,
\\
i \gamma^\mu D_\mu \psi 
&= g \phi^* \phi \psi, 
\endaligned
\ee
together with \eqref{eq:83GG}.
These PDEs can also be expressed as   
\bel{eq:euler-lagrange-u1-no-perturb} 
\aligned
\Box A^\nu - \del^\nu ( \text{div} A ) - 2 q^2  A^\nu \phi^* \phi
& = iq \big( \phi^* (\del^\nu \phi) - (\del^\nu \phi^*) \phi \big) + q \psi^* \gamma^0 \gamma^\nu \psi,
\\
\Box \phi - V'(\phi^* \phi) \phi - iq \phi \del_\mu A^\mu
& = 2iq A_\mu \del^\mu \phi + q^2 A^\mu A_\mu \phi  - g\phi \psi^* \gamma^0 \psi,
\\
i \gamma^\mu \del_\mu  \psi 
- g \phi^* \phi \psi
& = - q \gamma^\mu A_\mu \psi.
\endaligned
\ee

\begin{lemma} \label{lem-u1-gauge-symmetry}
$\Lcal$ has a $U(1)$-gauge symmetry.
\end{lemma}

\begin{proof}
The $U(1)$ gauge symmetry induces the following transformations
\bel{eq:U1-gauge-transf}
A_\mu \mapsto A_\mu ' 
:= 
A_\mu + \del_\mu \alpha, \quad 
\phi \mapsto \phi' 
:= 
e^{i q \alpha} \phi, \quad
\psi \mapsto \psi' 
:= 
e^{i q \alpha} \psi,
\ee
where $\alpha = \alpha (t, \bf{x})$ is some arbitrary function of space and time. This implies
\be 
\aligned
D_\mu \phi &\mapsto D_\mu ' \phi' 
= 
e^{i q \alpha} D_\mu \phi,
\\
D_\mu \psi & \mapsto D_\mu ' \psi' 
	= \del_\mu ( e^{i q \alpha} \psi) - i q ( A_\mu + \del_\mu \alpha ) e^{ i q \alpha} \psi = e^{i q \alpha} D_\mu \psi.
\endaligned
\ee
By the commutation of partial derivatives we immediately see that $F_{\mu \nu}$ is invariant under $U(1)$ gauge transformations.
Inserting these transformations into the Lagrangian one obtains
\be 
\aligned
\Lcal' 
&= -i \psi'^* \gamma^0  \gamma^\mu D'_\mu \psi'  -{1 \over 4} F'_{\mu \nu} F'^{\mu \nu} - (D'_\mu \phi')^* D' {}^\mu \phi' - V((\phi')^* \phi') + \phi '^* \phi' g\psi'^* \gamma^0 \psi'
\\
& = -i e^{-i q \alpha} \psi^* \gamma^0 \gamma^\mu \del_\mu ( e^{i q \alpha} \psi)
-i \big(e^{i q \alpha} \psi, -iq \gamma^\mu (A_\mu + \del_\mu \alpha) e^{i q \alpha} \psi) \big)
 -{1 \over 4}F_{\mu \nu} F^{\mu \nu} 
\\
	& \quad -(D_\mu \phi)^* e^{-i q \alpha}  e^{i q \alpha} D_\mu \phi - V(  \phi^* e^{-i q \alpha} e^{i q \alpha} \phi) + g \phi^* \phi e^{-i q \alpha} \psi^* \gamma^0 e^{i q \alpha} \psi
\\
& = - i \psi^* \gamma^0 \gamma^\mu D_\mu\psi -{1 \over 4} F_{\mu \nu} F^{\mu \nu} - (D {}_\mu \phi)^* D{}^\mu \phi - V(\phi^* \phi) + g\phi^* \phi \psi^* \gamma^0 \psi
= \Lcal.
\endaligned
\ee
Thus $\Lcal$ is invariant under the transformation \eqref{eq:U1-gauge-transf}. 
\end{proof}

 
\subsection{Propagation of an inhomogeneous Lorenz gauge}
Next, with a similar aim to that of Section \ref{subsec:dp-pde}, we will turn the PDE \eqref{eq:euler-lagrange-u1-no-perturb} into one of definite type by specifying a particular gauge for the vector field $A^\mu$. 
\begin{lemma}[The inhomogeneous Lorenz gauge] \label{lem:u1-gauge-propagate}
Let $X$ be a suitably regular scalar field. Consider the modified system 
\bel{eq:euler-lagrange-u1-no-perturb-mod} 
\aligned
\Box A^\nu- 2 q^2  A^\nu \phi^* \phi
& = 
\del^\nu X + iq \big( \phi^* (\del^\nu \phi) - (\del^\nu \phi^*) \phi \big) + q \psi^* \gamma^0 \gamma^\nu \psi,
\\
\Box \phi - V'(\phi^* \phi) \phi - iq \phi X
& = 
2iq A_\mu \del^\mu \phi + q^2 A^\mu A_\mu \phi  - g\phi \psi^* \gamma^0 \psi,
\\
i \gamma^\mu \del_\mu  \psi 
- g \phi^* \phi \psi
& = 
- q \gamma^\mu A_\mu \psi.
\endaligned
\ee
Suppose the initial data for \eqref{eq:euler-lagrange-u1-no-perturb-mod} satisfy
\bel{eq:u1-general-gauge}
\aligned
\big(\textrm{div}A - X \big) (t_0, \cdot) = 0,
\\ 
\del_t \big(\textrm{div}A - X \big) (t_0, \cdot) = 0.
\endaligned
\ee
Then as long as the solution to \eqref{eq:euler-lagrange-u1-no-perturb-mod} exists, it will satisfy $\div A =X$.
\end{lemma}

\begin{proof}
To propagate the gauge choice $\div A = X$ imposed on the initial data we will take the divergence of the first equation in \eqref{eq:euler-lagrange-u1-no-perturb-mod} and use the second evolution equation \eqref{eq:euler-lagrange-u1-no-perturb-mod}: 
$$
\aligned
\Box ( \text{div} A - X )
& = iq (\phi^* \Box \phi - \Box \phi^* \phi ) + 2 q^2 \del_\nu ( A^\nu \phi^* \phi ) + q \del_\nu (\psi^* \gamma^0 \gamma^\nu \psi)
\\
&= 
	iq \phi^* \Big( 2i q A_\mu \del^\mu \phi + iq \phi \del_\mu A^\mu
	+ q^2 A_\mu A^\mu \phi + V' \phi - g \phi \psi^* \gamma^0 \psi \Big)
\\
& \quad 
	- iq \Big( -2i q A_\mu \del^\mu \phi^* - iq \phi^* \del_\mu A^\mu 
	+ q^2 A_\mu A^\mu \phi^* + (V' \phi)^* - g \phi^* \psi^* \gamma^0 \psi \Big) \phi 
\\	
& \quad + 2q^2 \del_\nu ( A^\nu \phi^* \phi) + q \psi^* \gamma^0 \gamma^\mu \del_\mu \psi + q \del_\mu \psi^* \gamma^0 \gamma^\mu \psi
\\
& = iq \big(\phi^* V' \phi - (V' \phi)^* \phi \big) 
  = 0, 
\endaligned
$$
where in the last line we used the specific Higgs potential \eqref{eq:83GG}.
\end{proof}


\subsection{Gauge choice for the abelian model}

The Higgs potential has a non-zero minimum at 
$\phi_0 := v e^{i \theta/v}$ 
where
\be 
V(\phi_0^* \phi_0) 
= 0,
\ee
and $\theta: \RR^{1+3} \to \RR$ is arbitrary. There is ambiguity in this minimum state due to the $U(1)$ symmetry represented by $\theta$. A particular choice of $\theta$ will break this $U(1)$ symmetry, and such a scenario is termed `spontaneous symmetry breaking'. We consider perturbations of the form 
$$ 
\chi 
:= \phi - \phi_0,
$$
where $\phi_0$ is \textit{constant} in space and time and has magnitude $|\phi_0|=v$. 
The following result is a consequence of Lemma \ref{lem:u1-gauge-propagate} by choosing $X = -iq(\phi_0^* \chi - \chi^* \phi_0)$. 

\begin{lemma}  \label{lem:u1-linearised-gauge}
Suppose the initial data satisfy the following gauge condition
\bel{eq:u1-linearised-gauge} 
\aligned
\big(\textrm{div}A +iq(\phi_0^* \chi - \chi^* \phi_0) \big) (t_0, \cdot) 
= 0,
\\ 
\del_t \big(\textrm{div}A + iq(\phi_0^* \chi - \chi^* \phi_0) \big) (t_0, \cdot) 
= 0.
\endaligned
\ee
Then the Euler-Lagrange equations for \eqref{eq:lagrangian-u1} are equivalent to those written in \eqref{eq:u101}. 
\end{lemma}


\section{Structure and nonlinearity of the models} 
\label{sec:nonlinearities}

\subsection{Aim of this section}

In this section we will transform the variables in \eqref{eq:u101} in order to treat difficult terms in their nonlinearities. 
The nonlinearities, defined in \eqref{eq:u1-nl}, which require additional work to control are of the following two types
\bel{eq:w-w}
\psi^* \gamma^0 \gamma^\nu \psi,
\qquad
\psi^* \gamma^0 \psi,
\ee
appearing in $Q_{A^\nu}$ and $Q_\chi$ respectively. 
If the mass of the Dirac spinor is small or zero, that is $0 \leq m_g \ll \min (m_q, m_\lambda)$, then the nonlinearities in \eqref{eq:w-w} will have insufficiently fast decay for the bootstrap argument to close1. To address these issues, we employ a transformation introduced by \cite{Tsutsumi}, see \eqref{eq:def-B-vec} and \eqref{eq:def-tildechi}  below. 
Finally although the equation satisfied by $\chi$ is of ambiguous type, we find that $\chi$ can be decomposed into two components, with each component satisfying a Klein-Gordon equation.

 
\subsection{Hidden null structure from Tsutsumi}
Following Tsutsumi \cite{Tsutsumi} we can uncover a null structure using a particular transformation. Significantly however, we will not need to reduce to the `strong null forms' via any complicated transformations, as in \cite[eq (2.6)]{Tsutsumi2}. 

For complex-valued functions $\Phi(t,x), \Psi(t,x): \RR^{3+1} \to \CC^n$, recall the null form
\be 
Q_0(\Phi, \Psi)
:= (\del_0 \Phi)^* \del_0 \Psi - ( \nabla \Phi)^* \nabla \Psi.
\ee
Define a new variable 
\bel{eq:def-B-vec}
\tildeA^\nu 
:= A^\nu 
	+ {q \over m_q^2} \psi^* \gamma^0 \gamma^\nu \psi.
\ee
This satisfies the non-linear Klein-Gordon equation 
\bel{eq:full-B-pde}
\big( \Box - m_q^2 \big) \tildeA^\nu
= Q_{\tildeA^\nu},
\ee
in which the nonlinearities are
$$
\aligned
Q_{\tildeA^\nu}
&:=
    {2 q \over m_q^2} Q_0 (\psi, \gamma^0 \gamma^\nu \psi)
	+ {q \over m_q^2} G_\psi^* \gamma^0 \gamma^\nu \psi
	+ {q \over m_q^2} \psi^* \gamma^0 \gamma^\nu G_\psi
	+ {2 q  \over m_q^2} m_g^2\psi^* \gamma^0 \gamma^\nu \psi,
\\
G_\psi 
&:= - m_g Q_\psi - i \gamma^\nu \del_\nu Q_\psi
\\
& \, = 
	- m_g \big(  g (\phi_0^* \chi + \chi^* \phi_0 + \chi^* \chi) \psi -  q \gamma^\mu A_\mu \psi \big)
\\
& \quad
	- i g \gamma^\nu \del_\nu (\phi_0^* \chi + \chi^* \phi_0 + \chi^* \chi) \psi + i q \gamma^\nu \gamma^\mu \del_\nu A_\mu \psi
\\
& \quad	
	- g (\phi_0^* \chi + \chi^* \phi_0 + \chi^* \chi) (m_g \psi + Q_\psi)
	- 2ig A^\mu \del_\mu \psi - g A_\mu \gamma^\mu  (m_g \psi + Q_\psi),
\endaligned	
$$
where \eqref{eq:psi-second-order} and \eqref{eq:u1-nl} were used to compute $G_\psi$. 
Note the nonlinearity $\psi^* \gamma^0 \gamma^\nu \psi$ in \eqref{eq:full-B-pde} now appears with a good factor $m_g^2$. We will control this term using Propositon \ref{prop:energy-decop}. 


\subsection{Decomposition of $\chi$}

In order to study the behaviour of $\chi$, it is more convenient to consider equations for the following two variables
\be
\chi_\pm 
:= \phi_0^* \chi \pm \chi^* \phi_0.
\ee
Since the following identity holds
\bel{eq:chi-equivalent}
| \chi_+ |^2 + | \chi_- |^2
=
2 v^2 | \chi |^2,
\ee
it is equivalent to estimate either $\chi$ or $\chi_\pm$.
These new variables satisfy the Klein-Gordon equations \begin{align}
\label{eq:chi+}
\Box \chi_+ - m_\lambda^2 \chi_+
= 
Q_{\chi_+},
\\
\label{eq:chi-}
\Box \chi_- - m_q^2 \chi_-
= 
Q_{\chi_-},
\end{align}
with
\begin{align*}
Q_{\chi_+}
&:=
	2 i q A_\mu \del^\mu \chi_- + q^2 \chi_-^2 + q^2A^\mu A_\mu \big(2 v^2 + \chi_+ \big)
\\
&
\quad
	 +4 \lambda v^2 \chi^* \chi + 2 \lambda \chi_+^2 + 2 \lambda \chi_+ \chi^* \chi - g \big(2 v^2 + \chi_+ \big) \psi^* \gamma^0 \psi,
\\
Q_{\chi_-}
&:=
2 i q A_\mu \del^\mu \chi_+ + q^2 \chi_- \chi_+ + q^2 A_\mu A^\mu \chi_-
+
2 \lambda \chi_- \chi_+ + 2 \lambda \chi_- \chi^* \chi - g \chi_- \psi^* \gamma^0 \psi.
\end{align*}
Following Tsutsumi again \cite{Tsutsumi}, we define the new variable 
\bel{eq:def-tildechi}
\tildechi_+
:=
\chi_+ - {2 m_g \over m_\lambda^2} \psi^* \gamma^0 \psi,
\ee
which satisfies the following Klein-Gordon equation
\be 
\Box \tildechi_+ - m_\lambda^2 \tildechi_+
=
Q_{\tildechi_+}
\ee
with the nonlinearity
$$
Q_{\tildechi_+}
:=
-4 {m_g \over m_\lambda^2} Q_0(\psi, \gamma^0 \psi) + 2{m_g \over m_\lambda^2} G_\psi^* \gamma^0 \psi 
+ 2 {m_g \over m_\lambda^2} \psi^* \gamma^0 G_\psi - 4 { m_g^3 \over m_\lambda^2} \psi^* \gamma^0 \psi + Q_{\chi_+} + 2 m_g \psi^* \gamma^0 \psi.
$$
The final term $2 m_g \psi^* \gamma^0 \psi$ here cancels the problematic term in $Q_{\chi_+}$ and any other nonlinearities of the form $\psi^* \gamma^0 \psi$ now appear with a good factor of $m_g^3$ in front.

Note that in the case $0 \leq m_g \ll \min (m_q, m_\lambda)$, the second order formulation \eqref{eq:psi-second-order} of the Dirac equation is more like a nonlinear wave equation. In this case we do not have good bounds for either the $L^2_f$ or $L^\infty$ norm of $\psi$ and this is why the null structure of \eqref{eq:full-B-pde} and the factor $m_g^2$ in front of the $\psi$--$\psi$ interaction are needed. 
In the case $m_g \sim  \min (m_q, m_\lambda) $, better Klein-Gordon bounds are available and we do not require such a particular structure anymore.


\section{Bootstrap argument} 
\label{sec:bootstrap}

\subsection{Overview}

This section is devoted to using a bootstrap argument to prove Theorem \ref{thm:U1}. After the treatment of the equations in Section \ref{sec:nonlinearities}, we remind one that in the following analysis we will deal with the unknowns
\be 
\tildeA^\nu
=A^\nu + {q \over m_q^2} \psi^* \gamma^0 \gamma^\nu \psi,
\quad
\tildechi_+
=\chi_+ - {2 m_g \over m_\lambda^2} \psi^* \gamma^0 \psi,
\quad
\chi_-,
\quad
\psi,
\ee
which satisfy equations
$$
\aligned
\Box \tildeA^\nu - m_q^2 \tildeA^\nu
&= Q_{\tildeA^\nu},
\\
\Box \tildechi_+ - m_\lambda^2 \tildechi_+
&=
Q_{\tildechi_+}
\\
\Box \chi_- - m_q^2 \chi_-
&= 
Q_{\chi_-},
\\
i \gamma^\mu \del_\mu \psi - m_g \psi 
&= Q_{\psi}.
\endaligned
$$

We also provide here the following theorem, which details the $L^2$ and $L^\infty$ estimates of the unknowns. 

\begin{theorem} \label{thm:full-U(1)} 
Under the smallness assumptions in Theorem \ref{thm:U1} the solution satisfies the following energy estimates
\be
\aligned
\big\|(s/t) \del^I L^J \del_\mu (A_\nu,  \chi) \big\|_{L^2_f(\Hcal_s)} + \big\| (s/t) \del_\mu \del^I L^J (A_\nu, \chi) \big\|_{L^2_f(\Hcal_s)}
& \lesssim
	C_1 \eps, 
 \quad |I|+|J|\leq N,
\\
\big\|\del^I L^J (A_\nu, \chi) \big\|_{L^2_f(\Hcal_s)}
& \lesssim
	C_1 \eps, 
\quad |I|+|J|\leq N,
\\
\big\| (s/t) \del^I L^J \psi \big\|_{L^2_f(\Hcal_s)} 
& \lesssim
	C_1 \eps, 
\quad |I|+|J|\leq N,
\\	
\big\| (s/t) \del_\alpha \del^I L^J \psi \big\|_{L^2_f(\Hcal_s)} 
& \lesssim
	C_1 \eps \log s, 
\quad |I|+|J|= N,
\endaligned
\ee
and the following $L^\infty$ estimates
\be 
\aligned
\sup_{(t,x) \in \Hcal_s} 
	\Big(t^{1/2} s \big|\del_\alpha \del^I L^J (A_\nu, \chi) \big| + t^{1/2} s \big| \del^I L^J \del_\alpha  (A_\nu, \chi) \big| \Big) 
& \lesssim
	C_1 \eps, \quad |I| + |J| \leq N-2,
\\
\sup_{(t,x) \in \Hcal_s} \Big(t^{3/2} \big|\del^I L^J A_\nu, \del^I L^J \chi \big| \Big)
& \lesssim
	C_1 \eps, \quad |I| + |J| \leq N-2,
\\
\sup_{(t,x) \in \Hcal_s} \Big(t^{1/2} s \big|\del^I L^J \psi \big| \Big)
& \lesssim
	C_1 \eps, \quad |I| + |J| \leq N-2,	
\\
\sup_{(t,x) \in \Hcal_s} \Big(t^{1/2} s \big|\del_\alpha \del^I L^J \psi \big| \Big)
& \lesssim
	C_1 \eps \log s, \quad |I| + |J| = N-2,
\endaligned
\ee
where $C_1$ is some constant to be determined later.
\end{theorem}

\paragraph*{Strategy of the proofs of Theorems \ref{thm:U1}  and \ref{thm:full-U(1)}.}

The proofs are based on a bootstrap argument. In Section \ref{subsec:nullcom} we recall standard estimates for null terms, various commutators and Sobolev-type estimates on hyperboloids. The bootstrap assumptions will be made in \eqref{eq:bootsassumption0}. These bootstraps, combined with some standard commutator estimates and Sobolev-type inequalities, will lead to certain weak estimates in \eqref{eq:directL2} and \eqref{eq:directLinf}. In Section \ref{subsec:dirac909} we will use our first-order hyperboloidal energy to upgrade our estimates for the Dirac component, namely in Theorem \ref{thm:sup-Dirac1} and Corollary \ref{corol:sup-Dirac}. In Section \ref{subsec:refined909} we obtain estimates for the transformed variables $\tilde{A}^\nu$ and $\tilde{\chi}_+$ defined above. Putting all of this together we finally are able to close our bootstrap assumptions. 


\subsection{Standard Estimates: null forms, commutators and pointwise estimates} \label{subsec:nullcom}

We first illustrate estimates for the quadratic null terms, which, roughly speaking, reveal the following bounds 
$$
\big| \del^\alpha u \del_\alpha v \big|
\lesssim
\big(s^2 / t^2\big) \big| \del u \big| \big| \del v \big|.
$$

\begin{lemma}
We have the following estimate for the quadratic null term $T^{\alpha \beta} \del_\alpha u \del_\beta v$ with constants $T^{\alpha \beta}$ and $u,v$ sufficiently regular. 
\bel{eq:est-null1}
\aligned
&\quad
\big| \del^I L^J (T^{\mu \nu} \del_\mu u  \del_\nu  v) \big|
\\ 
&\lesssim 
\sum_{\substack{| I_1 | + | I_2 | \leq | I |, \\ |J_1| + |J_2| \leq |J|, \\ a, \beta}} 
   \Big( \big| \del^{I_1} L^{J_1} \underline{\del}_a u \big| \big| \del^{I_2} L^{J_2} \underline{\del}_\beta v  \big| 
       + \big| \del^{I_1} L^{J_1} \underline{\del}_\beta u | | \del^{I_2} L^{J_2} \underline{\del}_a v \big| \Big)
\\
&\quad
 + (s/t)^2 \sum_{\substack{| I_1 | + | I_2 | \leq | I |,\\ |J_1| + |J_2| \leq |J| }} \big|\del^{I_1} L^{J_1} \del_t u \big| \big| \del^{I_2} \L^{J_2} \del_t v \big|.
\endaligned
\ee
\end{lemma}
One refers to \cite{PLF-YM-book} for the proof.
We next recall the estimates for commutators, also proved in \cite{PLF-YM-book}. 
Heuristically speaking, the following lemma provides the relation
$$
\big| \Gamma \del^I L^J u \big|
\simeq
\big| \del^I L^J \Gamma u \big|,
$$
in which $\Gamma$ takes values from the set
$$
\{ \del_\alpha, \underline{\del}_\alpha, \del_\alpha \del_\beta, (s/t) \del_\alpha \}.
$$

\begin{lemma} \label{lem:est-comm}
Assume a function $u$ defined in the region $\mathcal{K}$ is regular enough, then we have
\begin{align}
\label{eq:est-cmt1}
\big| [\del^I L^J, \del_\alpha] u \big| 
&\leq 
C(| I |, |J|) \sum_{|J'|<|J|, \beta} \big|\del_\beta \del^I L^{J'} u \big|,
\\
\label{eq:est-cmt2}
\big| [\del^I L^J, \underline{\del}_a] u \big| 
&\leq 
C(| I |, |J|) \Big( \sum_{| I' |<| I |, | J' |< | J |, b} |\underline{\del}_b \del^{I'} L^{J'} u | + t^{-1} \sum_{| I' |\leq | I |, |J'|\leq |J|} | \del^{I'} L^{J'} u | \Big),
\\
\label{eq:est-cmt3}
\big| [\del^I L^J, \underline{\del}_\alpha] u \big| 
&\leq 
C(| I |, |J|) \Big( \sum_{| I' |<| I |, | J' |< | J |, \beta} \big|\del_\beta \del^{I'} L^{J'} u \big| 
+ t^{-1} \sum_{| I' |\leq | I |, |J'|\leq |J|, \beta} \big| \del_\beta \del^{I'} L^{J'} u \big| \Big),
\\
\label{eq:est-cmt4}
\big| [\del^I L^J, \del_\alpha \del_\beta] u \big| 
&\leq 
C(| I |, |J|) \sum_{| I' |\leq  | I |, |J'|<|J|, \gamma, \gamma '} \big| \del_\gamma \del_{\gamma '} \del^{I'} L^{J'} u \big|,
\\ 
\label{eq:est-cmt6}
\big| \del^I L^J ((s/t) \del_\alpha u) \big| 
& \leq 
\big|(s/t) \del_\alpha \del^I L^J u \big| 
+ C(| I |, |J|) \sum_{| I' |\leq  | I |, |J'|\leq |J|, \beta } \big|(s/t) \del_\beta \del^{I'} L^{J'} u \big|.
\end{align}
\end{lemma}

Furthermore we can obtain $L^\infty$ estimates by recalling the following Sobolev-type inequality on hyperboloids \cite{PLF-YM-book}.

\begin{proposition}[Sobolev-type inequality on hyperboloids]
\label{prop:Sobolev-type-inequality2}
For all sufficiently smooth functions $u= u(t, x)$ supported in the region $\{(t, x): |x|< t - 1\}$, then for $s\geq 2$ one has 
\bel{eq:Sobolev2}
\sup_{\Hcal_s} \big| t^{3/2} u(t, x) \big| 
\lesssim \sum_{| J |\leq 2} \big\| L^J u \big\|_{L^2_f(\Hcal_s)},
\ee
where the summation is over  Lorentz boosts $L$. Note the implied constant is uniform in $s \geq 2$, and one recalls that $t = \sqrt{s^2 + |x|^2}$ on $\Hcal_s$.
\end{proposition}

\begin{proof} We revisit here the proof from  \cite{PLF-YM-book}. Consider the function $w_s(x) := u(\sqrt{s^2+|x|^2},x)$. 
Fix $s_0$ and a point $(t_0,x_0)$ in $\Hcal_{s_0}$ (with
$t_0 = \sqrt{s_0^2 + |x_0|^2}$), and observe that
\bel{eq:gh6007}
\del_aw_{s_0}(x) = \underline{\del}_au\big(\sqrt{s_0^2+|x|^2},x\big) = \underline{\del}_au(t,x),
\ee
with $t = \sqrt{s_0^2 + |x|^2}$ and
$
t\del_aw_{s_0}(x) = t\underline{\del}_au\big(\sqrt{s_0^2+|x|^2},x\big) = L_au(t,x).
$
Then, introduce $g_{s_0,t_0}(y) := w_{s_0}(x_0 + t_0\,y)$ and write
$$
g_{s_0,t_0}(0) = w_{s_0}(x_0) = u\big(\sqrt{s_0^2+|x_0|^2},x_0\big)=u(t_0,x_0).
$$
From the standard Sobolev inequality applied to the function $g_{s_0,t_0}$, we get 
$$
\big|g_{s_0,t_0}(0)\big|^2\leq C\sum_{|I|\leq 2}\int_{B(0,1/3)}|\del^Ig_{s_0,t_0}(y)|^2 \, dy,
$$
$B(0, 1/3) \subset \RR^3$ being the ball centered at the origin with radius $1/3$.

In view of (with $x = x_0 + t_0y$)
$$
\aligned
\del_ag_{s_0,t_0}(y)
& = t_0\del_aw_{s_0}(x_0 + t_0y) 
\\
& = t_0\del_aw_{s_0}(x) = t_0\underline{\del}_au\big(t,x),
\endaligned
$$
in view of \eqref{eq:gh6007}, we have (for all $I$)
$\del^Ig_{s_0,t_0}(y) = (t_0\underline{\del})^I u(t,x)$ and, therefore, 
$$
\aligned
\big|g_{s_0,t_0}(0)\big|^2 \leq& C\sum_{|I|\leq 2}\int_{B(0,1/3)}\big|(t_0\underline{\del})^I u\big(t,x)\big)\big|^2dy
\\
= & C t_0^{-3}\sum_{|I|\leq 2}\int_{B((t_0,x_0),t_0/3)\cap \Hcal_{s_0}}\big|(t_0\underline{\del})^I u\big(t,x)\big)\big|^2dx.
\endaligned
$$

Note that
$$
\aligned
(t_0\underline{\del}_a(t_0\underline{\del}_b w_{s_0}))
& = t_0^2\underline{\del}_a\underline{\del}_bw_{s_0}
\\
& = (t_0/t)^2(t\underline{\del}_a)(t\underline{\del}_b) w_{s_0} - (t_0/t)^2 (x^a/t)L_b w_{s_0}
\endaligned
$$
and $x^a/t = x^a_0/t + yt_0/t = (x^a_0/t_0 + y)(t_0/t)$. In the region $y\in B(0,1/3)$,
the factor $|x^a / t|$ can always be bounded by 1, and thus (for $|I|\leq 2$)
    $$
    | (t \underline{\partial})^I u| \leq \sum_{|J| \leq |I|} |L^J u| (t_0
/ t)^{|I|}.
    $$

In the region $|x_0|\leq t_0/2$, we have $t_0\leq \frac{2}{\sqrt{3}}s_0$ so 
$$
t_0 \leq C s_0 \leq C \sqrt{|x|^2 + s_0^2} = Ct
$$
for some $C>0$.  When $|x_0|\geq t_0/2$, in the region $B((t_0,x_0),t_0/3)\cap \Hcal_{s_0}$ we get
$t_0 \leq C|x| \leq C\sqrt{|x|^2 + s_0^2} =Ct$ and thus 
$$
|(t_0 \underline{\del})^I u| \leq C \, \sum_{|J| \leq |I|} | L^J u|
$$
and
$$
\aligned
\big|g_{s_0,t_0}(y_0)\big|^2
\leq& Ct_0^{-3}\sum_{|I|\leq 2}\int_{B(x_0,t_0/3)\cap\Hcal_{s_0}}\big|(t\underline{\del})^I u\big(t,x)\big)\big|^2 \, dx
\\
\leq& Ct_0^{-3}\sum_{|I|\leq 2}\int_{\Hcal_{s_0}}\big|L^I u(t,x)\big|^2 \, dx. 
\endaligned
$$ 
\end{proof}

We will also make use of the following Gronwall inequality.

\begin{lemma}[Gronwall-type inequality]
\label{lem:Gronwall}
Let $u(t)$ be a non-negative function that satisfies the integral inequality
\be 
u(t) 
\leq C + \int_{t_0}^t b(s) u(s)^{1/2} \, d s, 
\quad C\geq 0,
\ee
where $b(t)$ is non-negative function for $t \geq t_0$. Then it holds
\be 
u(t) 
\leq C + \left( \int_{t_0}^{t} b(s) \, d s \right)^2.
\ee
\end{lemma}

\subsection{Bootstrap assumptions and basic estimates}
 
By the local well-posedness of semilinear PDEs, there exists an $s_1 > s_0$ in which the following bootstrap assumptions hold for all $s\in [s_0, s_1]$
\bel{eq:bootsassumption0}
\aligned
\Ecal_{m_q} \big(s, \del^I L^J A_\nu \big)^{1/2}
+
\Ecal_{m_\lambda} \big(s, \del^I L^J \chi \big)^{1/2}
&\leq C_1 \eps, &|I|+|J|\leq N,
\\
\Ecal_{m_g}(s, \del^I L^J \psi)^{1/2}
&\leq C_1 \eps,   &|I|+|J|\leq N-1,
\\
\Ecal_{m_g}(s, \del^I L^J \psi)^{1/2}
&\leq C_1 \eps \log s,  &|I|+|J|=N.
\endaligned
\ee
If we can prove the refined estimates
\bel{eq:bootrefined}
\aligned
\Ecal_{m_q} \big(s, \del^I L^J A_\nu \big)^{1/2}
+
\Ecal_{m_\lambda} \big(s, \del^I L^J \chi \big)^{1/2}
&\leq {1\over 2} C_1 \eps, &|I|+|J|\leq N,
\\
\Ecal_{m_g}(s, \del^I L^J \psi)^{1/2}
&\leq {1\over 2} C_1 \eps,   &|I|+|J|\leq N-1,
\\
\Ecal_{m_g}(s, \del^I L^J \psi)^{1/2}
&\leq {1\over 2} C_1 \eps \log s,  &|I|+|J|=N,
\endaligned
\ee
then we are able to assert that $s_1$ cannot be finite, which in turn implies a global existence result for \eqref{eq:intro-u1-pde-modified-lorenz}.

Combining the bootstrap assumptions \eqref{eq:bootsassumption0} with the estimates for commutators in Lemma \ref{lem:est-comm}, the following sets of estimates are obtained:
\bel{eq:directL2}
\aligned
\big\|(s/t) \del^I L^J \del_\mu (A_\nu,  \chi) \big\|_{L^2_f(\Hcal_s)} + \big\| (s/t) \del_\mu \del^I L^J (A_\nu, \chi) \big\|_{L^2_f(\Hcal_s)}
&\leq 
	C_1 \eps, &|I|+|J|\leq N.
\\
\big\|\del^I L^J (A_\nu, \chi) \big\|_{L^2_f(\Hcal_s)} 
&\leq 
	C_1 \eps, &|I|+|J|\leq N,
\\	
m_g \big\| \del^I L^J \psi \big\|_{L^2_f(\Hcal_s)}
+ \big\| (s/t) \del_\mu \del^I L^J \psi \big\|_{L^2_f(\Hcal_s)}
&\leq 
	C_1 \eps, &|I|+|J|\leq N-1,
\\
m_g \big\| \del^I L^J \psi \big\|_{L^2_f(\Hcal_s)}
+ \big\| (s/t) \del_\mu \del^I L^J \psi \big\|_{L^2_f(\Hcal_s)}
&\leq 
	C_1 \eps \log s, &|I|+|J|=N.
\endaligned
\ee

Combining these estimates with Proposition \ref{prop:Sobolev-type-inequality2} and Proposition \ref{prop:energy-decop} the following hold:
\bel{eq:directLinf}
\aligned
\sup_{(t,x) \in \Hcal_s} 
	\Big(t^{1/2} s \big|\del_\alpha \del^I L^J (A_\nu, \chi) \big| + t^{1/2} s \big| \del^I L^J \del_\alpha  (A_\nu, \chi) \big| \Big) 
& \lesssim
	C_1 \eps, \quad |I| + |J| \leq N-2,
\\
\sup_{(t,x) \in \Hcal_s} \Big(t^{3/2} \big|\del^I L^J A_\nu, \del^I L^J \chi \big| \Big)
& \lesssim
	C_1 \eps, \quad |I| + |J| \leq N-2,
\\
\sup_{(t,x) \in \Hcal_s} \Big( m_g t^{3/2} \big| \del^I L^J \psi \big|
+ t^{1/2} s \big|\del_\alpha \del^I L^J \psi, \del^I L^J \del_\alpha \psi \big| \Big)
& \lesssim
	C_1 \eps, \quad |I| + |J| \leq N-3,
\\
\sup_{(t,x) \in \Hcal_s} \Big( m_g t^{3/2} \big| \del^I L^J \psi \big|
+ t^{1/2} s \big|\del_\alpha \del^I L^J \psi, \del^I L^J \del_\alpha \psi \big| \Big)
& \lesssim
	C_1 \eps \log s, \quad |I| + |J| = N-2.
\endaligned
\ee


\subsection{First-order energy estimate for the Dirac field} \label{subsec:dirac909}
To obtain decay estimates for the Dirac component $\psi$, a standard method is to analyse the second-order form of the Dirac equation \eqref{eq:psi-second-order}. This is then a semilinear Klein-Gordon equation with mass $m_g^2$ and so there are now standard techniques to estimate the nonlinearity; see for example \cite{Alinhac} and \cite{PLF-YM-cmp}. However, the right-hand side term appearing in our wave equation \eqref{eq:psi-second-order} does not decay sufficiently fast for this argument to close, which is due to the possibly vanishing mass $m_g^2 \geq 0$. Thus at this point we recall Proposition \ref{eq:positive-energy-1storder} and the lower bound \eqref{eq:lowbound1} for the energy $E^\Hcal$. This motivates us to analyse the first-order form of the Dirac equation in the following Theorem to obtain certain improved $L^2$ and $L^\infty$ estimates for $\psi$ that are uniform in $m_g$. 

\begin{theorem} 
\label{thm:sup-Dirac1}
Under the same assumptions as Theorem \ref{thm:full-U(1)}, the Dirac field $\psi$ satisfies
\begin{align}
\big\| (s/t) \Lh^J \psi \big\|_{L^2_f(\Hcal_s)} 
& \lesssim
	\eps + (C_1 \eps)^2, \quad |J| \leq N, \label{eq:L2Dirac}
\\
\sup_{\Hcal_s} \big| t^{1/2} s \Lh^J \psi \big| 
& \lesssim 
	\eps + (C_1 \eps)^2, \quad |J|\leq N-2. 
\label{eq:supDirac}
\end{align} 
As a consequence, one has the following sup-norm estimate for $\psi$: 
\bel{eq:sup-Dirac}
\sup_{\Hcal_s} \big|t \del^I \Lh^J \psi \big| 
\lesssim \eps + (C_1 \eps)^2,
\qquad
| I | + |J| \leq N-2.
\ee
\end{theorem}

\begin{proof} 
{\bf Step 1.}
Recall the equation in \eqref{eq:u101} for the Dirac field 
$$
\aligned
&
i \gamma^\mu \del_\mu \psi - m_g \psi 
=
H \psi,
\\
&
H
:=
g (\phi_0^* \chi + \chi^* \phi_0 + \chi^* \chi) - g \gamma^\mu A_\mu. 
\endaligned
$$
Since $i \psi^* \gamma^0 H - i H^* \gamma^0 \psi = 0$ we have the following conserved energy
\be 
E^\Hcal (s, \psi) 
= E^\Hcal (s_0, \psi),
\ee
then according to the inequality \eqref{eq:lowbound1} in Proposition~\ref{eq:positive-energy-1storder}, we are able to initialise the induction argument by
$$ 
\| (s/t) \psi \|_{L^2_f(\Hcal_s)} \lesssim \eps.
$$

{\bf Step 2.}
For induction purposes, assume
$$
\big\| (s/t) \Lh^J \psi \big\|_{L^2_f(\Hcal_s)} 
\lesssim
\eps + (C \eps)^2
$$	
holds for $0 \leq |J| \leq k - 1 \leq N - 3$, and now consider the case $1 \leq |J| = k \leq N-2$.
Act $\Lh^J$ on the Dirac equation above to obtain
$$ 
\gamma^\mu \del_\mu (\Lh^J \psi) + i m_g \Lh^J \psi
= 
- i H (\Lh^J \psi) - i R,
$$
with 
$$ 
R
:= 
\Lh^J (H \psi) - H (\Lh^J \psi),
\qquad
H
=
g (\phi_0^* \chi + \chi^* \phi_0 + \chi^* \chi) - g \gamma^\mu A_\mu.
$$
Observe that $R$ contains only terms, up to some constant matrices, of type
\be
\Lh^{J_1} H \cdot \Lh^{J_2} \psi, \quad |J_1| + |J_2| \leq |J|, \quad |J_2| \leq k - 1.
\ee
Using \eqref{eq:1st-EE} and consequently by Lemma \ref{prop:Sobolev-type-inequality2} and the induction assumption
\be 
\aligned
E^\Hcal(s, \Lh^J \psi)^{1/2} 
&\leq  
E^\Hcal(s_0, \Lh^J \psi)^{1/2} + \int_{s_0}^s \| R \|_{L^2_f(\Hcal_{\bar{s}})} d \bar{s}
\\
& \lesssim 
\eps + \int_{s_0}^s \sum_{\substack{|J_1|+|J_2| \leq |J| \\ |J_2|\leq |J|-1}} \big\| (t/s) \Lh^{J_1} H \big\|_{L^\infty(\Hcal_{\bar{s}})} \big\|  (s/t) \Lh^{J_2} \psi \big\|_{L^2_f(\Hcal_{\bar{s}})} \Big) \, d \bar{s}
\\
& \lesssim \eps + (C_1 \eps)^2 \int_{s_0}^s \bar{s}^{-3/2} \, d \bar{s},
\endaligned
\ee
which gives
$$
\big\| (s/t) \Lh^J \psi \big\|_{L^2_f(\Hcal_s)} 
\lesssim \eps, 
\quad |J| \leq k.
$$

{\bf Step 3.}
The above analysis shows for $|J| \leq N-2$
$$ 
\sum_{|J'| \leq 2} \big\| (s/t) \Lh^{J'} \Lh^J \psi \big\|_{L^2_f(\Hcal_s)} 
\lesssim 
\eps + (C_1 \eps)^2.
$$
Thus by the Sobolev inequality \eqref{eq:Sobolev2}, we deduce
$$ 
\sup_{\Hcal_s} \big|t^{1/2} s \Lh^J \psi \big| 
\lesssim
\epsilon + (C_1 \eps)^2, 
\quad |J|\leq N-4.
$$

{\bf Step 4.}
We now consider the case $|J|=N-1$. An energy estimate yields
\be 
\aligned
E^\Hcal(s, \Lh^J \psi)^{1/2}
&\leq 
\eps + \int_{s_0}^s \Big( \sum_{\substack{|J_1|+|J_2| \leq N-1,\\ |J_1|> |J_2|, |J_2|\leq N-4}} \big\| \Lh^{J_1} H \big\|_{L^2_f(\Hcal_{\bar{s}})} \big\| \Lh^{J_2} \psi \big\|_{L^\infty(\Hcal_{\bar{s}})} 
 \\
& \quad + \sum_{\substack{|J_1|+|J_2| \leq N-1,\\ |J_1|\leq |J_2|, |J_2|\leq N-2}} \big\| (t/s) \Lh^{J_1} H \big\|_{L^\infty(\Hcal_{\bar{s}})} \big\| (s/t) \Lh^{J_2} \psi \big\|_{L^2_f(\Hcal_{\bar{s}})} \Big) \, d \bar{s}
\\
&\leq 
\eps + (C_1 \eps)^2 \int_{s_0}^s \bar{s}^{-3/2} \, d \bar{s},
\endaligned
\ee
which implies
\be 
\big\| (\bar{s}/t) \Lh^J \psi \big\|_{L^2_f(\Hcal_{\bar{s}})} 
\lesssim 
\eps + (C_1 \eps)^2.
\ee
The same analysis also applies to the case $|J|=N$. And repeating Step 3 gives \eqref{eq:supDirac} for $|J|\leq N-2$.
\end{proof}

As a consequence, we have the following sup-norm estimates for $\psi$.

\begin{corollary} \label{corol:sup-Dirac}
It holds that
\bel{eq:sup-Dirac-mm}
\sup_{\Hcal_s} \big| t^{1/2} s \del^I L^J \psi \big| 
\lesssim \eps + (C_1 \eps)^2,
\qquad
| I | + |J| \leq N-2.
\ee
\end{corollary}


\subsection{Refined estimates} \label{subsec:refined909}

In this final subsection we close our bootstrap argument. For this to work we move to the transformed vector field $\tildeA^\nu$ defined in \eqref{eq:def-B-vec} and the transformed scalar field $\tildechi_+$ defined in \eqref{eq:def-tildechi}, which are heuristically of the form
$$ 
A^\nu 
= 
\tildeA^\nu + \Ocal(|\psi|^2),
\qquad
\chi_+
=
\tildechi_+ + \Ocal(|\psi|^2).
$$
Thus using the estimates for $A^\nu$ and $\chi_+$ coming from \eqref{eq:directL2} and \eqref{eq:directLinf}, together with the previous energy and sup-norm estimates for $\psi$, the following estimates for $\tildeA^\nu$ and $\tildechi_+$ hold
\bel{eq:est-B}
\aligned
	\big\|(s/t) \del^I L^J \del_\mu (\tildeA_\nu, \tildechi_+) \big\|_{L^2_f(\Hcal_s)} 
	+ \big\| (s/t) \del_\mu \del^I L^J (\tildeA_\nu, \tildechi_+) \big\|_{L^2_f(\Hcal_s)}
&\lesssim 
	C_1 \eps, &|I|+|J|\leq N,
\\
\big\|\del^I L^J (\tildeA_\nu, \tildechi_+) \big\|_{L^2_f(\Hcal_s)}
&\lesssim 
	C_1 \eps, &|I|+|J|\leq N,
\\
\sup_{\Hcal_s} \Big( t^{1/2} s \big| \del_\alpha \del^I L^J (\tildeA_\nu, \tildechi_+) \big| + t^{1/2} s \big|\del^I L^J \del_\alpha (\tildeA_\nu, \tildechi_+) \big| \Big)
&\lesssim 
	C_1 \eps, &|I|+|J|\leq N-2,
\\
\sup_{\Hcal_s} \Big(t^{3/2} \big| \del^I L^J (\tildeA_\nu, \tildechi_+) \big| \Big)
&\lesssim 
	C_1 \eps, &|I|+|J|\leq N-2.
\endaligned
\ee 
We first look at the energy for $\psi$ in the case $|I|+|J|=N$
$$ 
\mE_{m_g} (s, \del^I L^J \psi)^{1/2} 
\leq \eps + C \sum_{\nu, \mu} \int_{s_0}^s \big( \big\| \del^I L^J ((\del_\mu H )\psi) \big\|_{L^2_f(\Hcal_{\bar{s}})}  + \big\| \del^I L^J ( H \del_\mu \psi) \big\|_{L^2_f(\Hcal_{\bar{s}})}  \big) \, d \bar{s}.
$$
By noting
\be
\aligned
&\quad
\sum_{\mu}  \| \del^I L^J (\del_\mu H \psi) \|_{L^2_f(\Hcal_{\bar{s}})}
\\
&\leq 
	\sum_{\substack{I_1+I_2=I, J_1+J_2=J\\ |I_1|+|J_1|\leq N-2, \, \mu}} \big\|(t/s) \del^{I_1} L^{J_1} \del_\mu H \big\|_{L^\infty(\Hcal_{\bar{s}})} \big\|(s/t) \del^{I_2}L^{J_2} \psi \big\|_{L^2_f(\Hcal_{\bar{s}})} 
\\
& \quad
	+ \sum_{\substack{I_1+I_2=I, J_1+J_2=J\\ |I_1|+|J_1|\geq N-1, \, \mu}} \big\|(s/t) \del^{I_1} L^{J_1} \del_\mu H \big\|_{L^2_f(\Hcal_{\bar{s}})} \big\|(t/s) \del^{I_2}L^{J_2} \psi \big\|_{L^\infty(\Hcal_{\bar{s}})}
\\
&\lesssim 
	(C_1 \eps)^2 \bar{s}^{-3/2} \log \bar{s} + (C_1 \eps)^2 \bar{s}^{-1}
       \lesssim 
	(C_1 \eps)^2 \bar{s}^{-1},
\endaligned
\ee
we obtain
\be
\mE_{m_g}(s, \del^I L^J \psi)^{1/2}
\leq \eps + C (C_1 \eps)^2 \log s.
\ee
Similarly for $|I|+|J|\leq N-1$ we obtain
\be 
\mE_{m_g}(s, \del^I L^J \psi)^{1/2} 
\leq \eps + C (C_1 \eps)^2.
\ee

In order to obtain estimates for $A^\nu$, we first bound the energy for $\tildeA^\nu$
\be 
\aligned
\mE_m(s, \del^I L^J \tilde A^\nu)^{1/2}
&\leq 
	\eps + \int_{s_0}^s \big\|\del^I L^J ( \la \gamma^0 \gamma^\nu \del_\mu \psi, \del^\mu \psi \ra ) \big\|_{L^2_f(\Hcal_{\bar{s}})} 
	+ m_g^2 \big\| \del^I L^J (\psi^* \gamma^0 \gamma^\nu \psi ) \big\|_{L^2_f(\Hcal_{\bar{s}})} 
\\
& \quad
	+ \big\|\del^I L^J (\psi ^2 \del A + \psi A \del \psi +  \psi^3 \del \psi ) \big\|_{L^2_f(\Hcal_{\bar{s}})}  \, d \bar{s}
\\
&\leq 
\eps + C (C_1 \eps)^2.
\endaligned
\ee
Next, recalling definition \eqref{eq:def-B-vec} we use Young's inequality to obtain for all $|I|+|J|\leq N$
\be 
\aligned
\mE_{m_q}(s, \del^I L^J  A^\nu)^{1/2}
& \leq (3/2) \mE_{m_q} (s, \del^I L^J \tildeA^\nu)^{1/2} + (3/2) \mE_{m_q} (s, \del^I L^J (\psi)^2)^{1/2}
\\
& \leq (3/2) \eps + C \sum_{\substack{I_1+I_2=I, J_1+J_2=J \\ I_1 + J_1 \leq N-2}} \big\| \del^{I_1} L^{J_1} \psi \big\|_{L^\infty(\Hcal_s)}
\mE(s, \del^{I_2} L^{J_2} \psi)^{1/2}
\\
& \lesssim \eps + (C_1 \eps)^2.
\endaligned
\ee
A similar procedure gives the refined estimates for $\chi_+$
\bel{eq:refine-chi+} 
\Ecal_{m_\lambda} (s, \del^I L^J \chi_+)^{1/2}
\leq 
(3/2) \eps + C (C_1 \eps)^2,
\ee
while the refined estimates for $\chi_-$
\bel{eq:refine-chi-} 
\Ecal_{m_q} (s, \del^I L^J \chi_-)^{1/2}
\leq 
\eps + C (C_1 \eps)^2
\ee
can be obtained directly.
A combination of \eqref{eq:refine-chi+} and \eqref{eq:refine-chi-} gives the refined estimates for $\chi$
\be 
\Ecal_{m_\lambda} (s, \del^I L^J \chi)^{1/2}
\lesssim
\eps + (C_1 \eps)^2.
\ee
By choosing $C_1$ sufficiently large and $\eps$ sufficiently small, we arrive at the refined bounds \eqref{eq:bootrefined}. This shows global existence and thus completes the proof of Theorem \ref{thm:full-U(1)}. Furthermore recalling the relation $t < s^2$ within the cone
$\mathcal{K}$, then \eqref{eq:directLinf} and \eqref{eq:sup-Dirac-mm} verify \eqref{eq:mainthmpsiest}.


\subsection*{Acknowledgements}

The authors were supported by the Innovative Training Networks (ITN) grant 642768, entitled ModCompShock. The second author (PLF) 
also gratefully acknowledge support from the Simons Center for Geometry and Physics, Stony Brook University, at which some of the research for this paper was performed. This work was completed when the third author (ZW) visited Sorbonne Universit\'e for several months in 2017--2018. During the completion of this work, the third author also acknowledges financial support from the Austrian Science Fund (FWF) project P29900-N27 `Geometric Transport equations and the non-vacuum Einstein-flow'. 



\begin{thebibliography}{10}

\bibitem{Ait-Hey} \auth{I. Aitchison and A. Hey},
Gauge theories in particle physics: a practical introduction, Vol. 1, CRC Press, (2012)

\bibitem{Alinhac} \auth{S. Alinhac}, 
Semi-linear hyperbolic systems with blow-up at infinity,
Indiana Univ. Math. J. 55 (2006), 1209--1232. 

\bibitem{Bachelot} \auth{A. Bachelot,}
Probl\`eme de Cauchy global pour des syst\`emes de Dirac-Klein-Gordon, 
Ann. Inst. Henri Poincar\'e 48 (1988), 387--422. 

\bibitem{Bournaveas} 
\auth{N. Bournaveas,}
Local existence of energy class solutions for the Dirac-Klein-Gordon equations, 
Comm. Part. Differential Equa. 24 (1999), 1167--1193. 

\bibitem{Ch-ChoBru} 
\auth{Y. Choquet-Bruhat and D. Christodoulou,}
Existence of global solutions of the Yang-Mills, Higgs and spinor field equations in $3+1$ dimensions,
Ann. Sci. \'Ecole Norm. Sup. 4 (1981), 481--506.

\bibitem{Dong19} 
{\sc S. Dong,}
The zero mass problem for Klein-Gordon equations,
Preprint ArXiv:1905.08620. 

\bibitem{D-F-S} \auth{P. D’Ancona, D. Foschi, and S. Selberg,}
Null structure and almost optimal local regularity of the Dirac–Klein–Gordon system, 
J. Eur. Math. Soc. 9 (2007), no. 4, 877--899.

\bibitem{Fajman} 
{\sc D. Fajman, J. Joudioux, and J. Smulevici,}
The stability of the Minkowski space for the Einstein-Vlasov system,
Preprint ArXiv:1707.06141. 

\bibitem{Friedrich1}
\auth{H. Friedrich,}
On the regular and the asymptotic characteristic initial value problem for Einstein’s vacuum field equations, 
Proc. R. Soc. London Ser. A 375 (1981), 169--184.

\bibitem{Friedrich2}
\auth{H. Friedrich,} 
Cauchy problems for the conformal vacuum field equations in general relativity, 
Commun. Math. Phys. 91 (1983), 445--472.

\bibitem{Georgiev} \auth{V. Georgiev,} 
Global solution of the system of wave and Klein--Gordon equations, 
Math. Z. 203 (1990), 683--698.

\bibitem{Hormander}
\auth{L. H\"ormander,}
{\sl Lectures on nonlinear hyperbolic differential equations,}
Springer Verlag, Berlin, 1997.

\bibitem{Katayama12a} 
\auth{S. Katayama,}
Global existence for coupled systems of nonlinear wave and Klein-Gordon equations in three space dimensions,
Math. Z. 270 (2012), 487--513.

\bibitem{Katayama12b}
\auth{S. Katayama,}
Asymptotic pointwise behavior for systems of semilinear wave equations in three space dimensions,
J. Hyperbolic Differ. Equ. 9 (2012), 263--323.

\bibitem{Klainerman80}
\auth{S. Klainerman,}
Global existence for nonlinear wave equations,
Comm. Pure Appl. Math. 33 (1980), 43--101.

\bibitem{Klainerman85}
\auth{S. Klainerman,}
Global existence of small amplitude solutions to nonlinear Klein-Gordon equations in four spacetime dimensions, 
Comm. Pure Appl. Math. 38 (1985), 631--641.

\bibitem{PLF-YM-book} \auth{P.G. LeFloch and Y. Ma,}
{\sl The hyperboloidal foliation method,}  World Scientific Press, 2014. 

\bibitem{PLF-YM-cmp} \auth{P.G. LeFloch and Y. Ma,}
The global nonlinear stability of Minkowski space for self-gravitating massive fields. 
The wave-Klein-Gordon model, 
Comm. Math. Phys. 346 (2016), 603--665.   

\bibitem{PLF-YM-book2} \auth{P.G. LeFloch and Y. Ma,} 
{\sl The global nonlinear stability of Minkowski space for self-gravitating massive fields,} 
World Scientific Press, Singapore, 2017. 

\bibitem{PLF-YM-arXiv1} \auth{P.G. LeFloch and Y. Ma,} 
The global nonlinear stability of Minkowski space. Einstein equations, f(R)-modified gravity, and Klein-Gordon fields,
Preprint arXiv:1712.10045.


\bibitem{PLF-YM-Comptes} \auth{P.G. LeFloch and Y. Ma,}
The global nonlinear stability of Minkowski space for the Einstein equations in the presence of a massive field. 
Comptes Rendus Mathematique. 354 (2016). 

\bibitem{PLF-Wei} \auth{P.G. LeFloch and C.-H. Wei,}    
Boundedness of the total energy of relativistic membranes evolving in a curved spacetime, 
J. Differential Equations 265 (2018), 312--331.


\bibitem{Smulevici} 
\auth{J. Smulevici,} 
Small data solutions of the Vlasov-Poisson system and the vector field method, 
Ann. PDE 11 (2016), 11--66.

\bibitem{Tsutsumi} \auth{Y. Tsutsumi,} 
Global solutions for the Dirac-Proca equations with small initial data in $3+ 1$ spacetime dimensions, 
J. Math. Anal. Appl. 278 (2003), 485--499. 

\bibitem{Tsutsumi2} \auth{Y. Tsutsumi,}
Stability of constant equilibrium for the Maxwell--Higgs equations, Funkcial. Ekvac. 46 (2003), 41--62.

\bibitem{Tzvetkov} \auth{N. Tzvetkov,}
Existence of global solutions to nonlinear massless Dirac system and wave equation with small data, 
Tsukuba J. Math. 22 (1998), 193--211.

\end{thebibliography}
\end{document}